%% file: BBMPV_EnriquesTranscendental.tex
\documentclass[12pt, reqno]{amsart}

\usepackage{verbatim}
\usepackage{amssymb, amsmath, amsthm}
\usepackage{hyperref}
\usepackage[alphabetic,lite]{amsrefs}
\usepackage{amscd}   
\usepackage{fullpage}
\usepackage[all]{xy} 
\usepackage{pdflscape}
\DeclareFontEncoding{OT2}{}{} 

\usepackage[usenames, dvipsnames]{color}

\usepackage{graphicx}
\usepackage{tikz}
\usetikzlibrary{matrix,arrows,decorations.pathmorphing}

\usepackage{tikz-cd}

\newtheorem{lemma}{Lemma}[section]
\newtheorem{theorem}[lemma]{Theorem}

\newtheorem{prop}[lemma]{Proposition}

\newtheorem{claim*}{Claim}
\newtheorem{thm}[lemma]{Theorem}

\newtheorem*{theorem*}{Theorem}

\theoremstyle{definition}
\newtheorem{remark}[lemma]{Remark}

\newcommand{\A}{{\mathbb A}}
\newcommand{\G}{{\mathbb G}}

\newcommand{\PP}{{\mathbb P}}

\newcommand{\F}{{\mathbb F}}
\newcommand{\Q}{{\mathbb Q}}

\newcommand{\Z}{{\mathbb Z}}

\newcommand{\Qbar}{{\overline{\Q}}}

\newcommand{\kbar}{{\overline{k}}}

\newcommand{\Ybar}{{\overline{Y}}}

\newcommand{\kk}{{\mathbf k}}

\newcommand{\calA}{{\mathcal A}}
\newcommand{\calB}{{\mathcal B}}

\newcommand{\calE}{{\mathcal E}}

\newcommand{\frakS}{{\mathfrak S}}

\usepackage[mathscr]{euscript}

\DeclareMathOperator{\HH}{H}

\DeclareMathOperator{\im}{im}

\DeclareMathOperator{\Gal}{Gal}

\DeclareMathOperator{\Cor}{Cor}

\DeclareMathOperator{\Norm}{Norm}
\DeclareMathOperator{\Br}{Br}

\DeclareMathOperator{\divv}{div}

\DeclareMathOperator{\Div}{Div}
\DeclareMathOperator{\Pic}{Pic}
\DeclareMathOperator{\Jac}{Jac}

\DeclareMathOperator{\Spec}{Spec}

\DeclareMathOperator{\sep}{sep}

\DeclareMathOperator{\et}{\textrm{\normalfont \'et}}

\DeclareMathOperator{\Princ}{Princ}

\DeclareMathOperator{\NS}{NS}

\newcommand{\isom}{\cong}

\numberwithin{equation}{section}
\numberwithin{table}{section}

\newcommand{\defi}[1]{\textsf{#1}} 

\newcommand{\BlY}{\nu}
\newcommand{\BlX}{{\tilde\BlY}}
\newcommand{\BldP}{{\BlY_\dP}}
\newcommand{\tildeBldP}{{\BlX_\dP}}
\newcommand{\dP}{S}
\newcommand{\opendP}{V}
\newcommand{\quadric}{{S^0}}
\newcommand{\Xabc}{{X_{\mathbf{a}}}}
\newcommand{\Yabc}{{Y_{\mathbf{a}}}}
\newcommand{\Xabcbar}{{\overline{X}_{\mathbf{a}}}}
\newcommand{\Yabcbar}{{\overline{Y}_{\mathbf{a}}}}
\newcommand{\Ctilde}{{\widetilde{C}}}
\newcommand{\Dtilde}{{\widetilde{D}}}
\newcommand{\Btilde}{{\widetilde{B}}}
\newcommand{\Stilde}{{\widetilde{S}}}
\newcommand{\xminusalpha}{j}
\newcommand{\To}{\longrightarrow}
\newcommand{\Sbar}{\overline{S}}
\newcommand{\Ubar}{\overline{U}}

\title{Insufficiency of the Brauer-Manin obstruction for Enriques surfaces}
\author{Francesca Balestrieri}
\address{University of Oxford, Mathematical Institute,  Oxford, OX2 6HD, United Kingdom}
\email{balestrieri@maths.ox.ac.uk}
\urladdr{http://www.maths.ox.ac.uk/people/francesca.balestrieri}

\author{Jennifer Berg}
\address{The University of Texas at Austin, Department of Mathematics, 2515 Speedway, RLM 8.100, Austin, TX 78712, USA}
\email{jberg@math.utexas.edu}
\urladdr{http://ma.utexas.edu/users/jberg}

\author{Michelle Manes}

\address{University of Hawai\kern.05em`\kern.05em\relax i\ at M\=anoa, Department of Mathematics, 2565 McCarthy Mall Keller 401A, Honolulu, HI 96822, USA}
\email{mmanes@math.hawaii.edu}
\urladdr{http://math.hawaii.edu/\~{}mmanes}

\author{Jennifer Park}

\address{McGill University, Department of Mathematics, 845 Rue Sherbrooke Ouest 
Montr\'eal, QC, H3A 0G4, Canada}
\email{jennifer.park2@mcgill.ca}
\urladdr{http://www.math.mcgill.ca/jpark/}

\author{Bianca Viray}

\address{University of Washington, Department of Mathematics, Box 354350, Seattle, WA 98195, USA}
\email{bviray@math.washington.edu}
\urladdr{http://math.washington.edu/\~{}bviray}

\thanks{F.B. supported by EPSRC scholarship EP/L505031/1. M.M. partially supported by NSF grant  DMS-1102858. J.P. partially supported by NSERC PDF grant. B.V. partially supported by NSF grant DMS-1002933}

\date{}

\begin{document}

	\begin{abstract}
		In~\cite{VAV-Enriques}, V\'arilly-Alvarado and the last author constructed an Enriques surface $X$ over $\Q$ with an \'etale-Brauer obstruction to the Hasse principle and no \emph{algebraic} Brauer-Manin obstruction.  In this paper, we show that the nontrivial Brauer class of $X_{\Qbar}$ does not descend to $\Q$.  Together with the results of~\cite{VAV-Enriques}, this proves that the Brauer-Manin obstruction is insufficient to explain all failures of the Hasse principle on Enriques surfaces.  
		
		The methods of this paper build on the ideas in~\cites{CV-BrauerSurfaces, CV-BrauerCurves, IOOV-UnramifiedBrauer}: we study geometrically unramified Brauer classes on $X$ via pullback of ramified Brauer classes on a rational surface.  Notably, we develop techniques which work over fields which are not necessarily separably closed, in particular, over number fields.  
	\end{abstract}

	\maketitle
	

	\input{Introduction}
	\section*{Acknowledgements}
		This project began at the Women in Numbers 3 conference at the Banff International Research Station.  We thank BIRS for providing excellent working conditions and the organizers of WIN3, Ling Long, Rachel Pries, and Katherine Stange, for their support.  We also thank Anthony V\'arilly-Alvarado for allowing us to reproduce some of the tables from~\cite{VAV-Enriques} in this paper for the convenience of the reader.
    \input{PullbackBrauer}

	\input{GeometryDoubleCoverings}
	\input{TranscendentalRepresentative}
	\input{ProofOfMainTheorem}




\appendix\label{app:Fields}
\section{Fields, defining equations, and Galois actions}
	The splitting field $K$ of the genus $1$ curves $C_i, \Ctilde_i, \Dtilde_i, F_i, G_i$ is generated by 
	\begin{align}
		i, \sqrt2, \sqrt5, \sqrt{a}, \sqrt{c}, \eta_0 := \sqrt{c^2 - 100ab}, 
		\gamma_0:= \sqrt{-c^2 - 5bc - 10ac - 25ab},\label{eqn:gens1}\\
		 \sqrt[4]{ab}, \sqrt{-2 + 2\sqrt{2}},
		\sqrt{-c-10\sqrt{ab}}, \label{eqn:gens2}\\
		\theta_0 := \sqrt{4a^2 + b^2},\; \xi_0 := \sqrt{a + b + c/5}, \;	
		\xi_0':= \sqrt{a + b/4 + c/10}, \label{eqn:gens3}\\
		\theta_1^+ := \sqrt{20a^2 -10ab - 2bc + (10a + 2c)\theta_0},\; 
		\theta_2^+ := \sqrt{-5a - 5/2b - 5/2\theta_0}, \label{eqn:gens4}\\
		\xi_1^+ := \sqrt{20a + 10b + 3c + 20\xi_0\xi_0'},\;
		\xi_2^+ := \sqrt{4a + 2b + 2/5c + 4\xi_0\xi_0'}.\label{eqn:gens5}
	\end{align}
    Define
    \begin{equation}
        \begin{array}{rclcrcl}
        \eta_1^+ & := &  
            \frac{c - \eta_0 + 10\sqrt{ab}}
            {10\sqrt{a}\sqrt{-c - 10\sqrt{ab}}} &  &
        \gamma_1^+ & := &
            (\theta_1^+)^{-1}\left(10a^2 - 5ab - bc + 2a\gamma_0 
            + (c + 5a)\theta_0\right) \\
        \eta_1^- & := &  
            \frac{c + \eta_0 + 10\sqrt{ab}}
            {10\sqrt{a}\sqrt{-c - 10\sqrt{ab}}}  &  &
        \gamma_1^- & := &
            (\theta_1^+)^{-1}\left(10a^2 - 5ab - bc - 2a\gamma_0 
            + (c + 5a)\theta_0\right). \\
        \end{array}
    \end{equation}
    The following subfields of $K$ are of particular interest:
    \begin{equation}\label{eq:fields}
        \Q \;\subset \;
        K_0 := \Q\left(i, \sqrt{2},\sqrt5,\sqrt{-2 + 2\sqrt2}\right) \;\subset\;
        K_1 := K_0\left(\theta_0, \sqrt{ab}, \eta_1^+, \gamma_1^+\right)
        \;\subset\; K.
    \end{equation}
The field extensions $K/\Q$ and $K_1/\Q$ are Galois, as the following relations show:
	\begin{align*}
		\sqrt{-2 - 2\sqrt{2}} = \frac{2i}{\sqrt{-2 + 2\sqrt{2}}},\quad &
		\sqrt{-c + 10\sqrt{ab}} = 
			\frac{\eta_0}{\sqrt{-c - 10\sqrt{ab}}},\\
		\theta_1^- := \sqrt{20a^2 -10ab - 2bc + (10a + 2c)\theta_0} = \frac{4a\gamma_0}{\theta_1^+}, \quad&
		 \theta_2^- := \sqrt{-5a - 5/2b + 5/2\theta_0} = \frac{5\sqrt{ab}}{\theta_2^+},\\
		\xi_1^- := \sqrt{20a + 10b + 3c - 20\xi_0\xi_0'} = \frac{\eta_0}{\xi_1^+}, \quad &		\xi_2^- := \sqrt{4a + 2b + 2/5c - 4\xi_0\xi_0'} = \frac{2\gamma_0}{5\xi_2^+},\\
        (\gamma_1^+)^2  =  10a^2 - 5ab - bc + 2a\gamma_0,\quad &
        (\eta_1^+)^2  = \frac{-c + \eta_0}{50a},\\
        (\gamma_1^-)^2  =  10a^2 - 5ab - bc - 2a\gamma_0,\quad &
        (\eta_1^-)^2  = \frac{-c - \eta_0}{50a},\\
        \gamma_1^+\gamma_1^- = (5a + c)\theta_0, \quad &
        \eta_1^+\eta_1^- = \frac{-1}{5a}\sqrt{ab}.
	\end{align*}

    Tables~\ref{table:WeierstrassPoints} and~\ref{table:curves} show that these fields have the properties claimed in~\S\ref{sec:Geometry}.
    
    \begin{remark}
    Tables~\ref{table:WeierstrassPoints} and~\ref{table:curves} list defining equations of a subvariety of the $\Yabc$.  The image of this subvariety gives the corresponding object in $\Xabc$.
    \end{remark}
    
    \renewcommand{\arraystretch}{1.2}
    \begin{table}[!b]
        \begin{tabular}{|c|rl||c|rl|}
            \hline
            $\tilde{B}$ & \multicolumn{2}{|c||}{Defining equations} &
            $\tilde{B}$ & \multicolumn{2}{|c|}{Defining equations} \\
            \hline
            $P_1$ 
                & $10av_0 - (c - \eta_0)v_1$,
                & $v_2 - \eta_1^+v_1$ &
            $Q_1$
               & $(c+ 5a)v_0 + (c - \gamma_0)v_1$,
               & $(c + 5a)v_2 - \gamma_1^+v_1$ \\
            \hline
             $P_2$ 
                & $10av_0 - (c - \eta_0)v_1$,
                & $v_2 + \eta_1^+v_1$ &
            $Q_2$
               & $(c+ 5a)v_0  + (c - \gamma_0)v_1$,
               & $(c + 5a)v_2 + \gamma_1^+v_1$ \\
            \hline
             $P_3$ 
                & $10av_0 - (c + \eta_0)v_1$,
                & $v_2 - \eta_1^-v_1$ &
            $Q_3$
                & $(c+ 5a)v_0  + (c + \gamma_0)v_1$,
                & $(c + 5a)v_2 - \gamma_1^-v_1$ \\
            \hline
             $P_4$ 
                & $10av_0 - (c + \eta_0)v_1$,
                & $v_2 + \eta_1^-v_1$ &
            $Q_4$
               & $(c+ 5a)v_0  + (c + \gamma_0)v_1$,
               & $(c + 5a)v_2 + \gamma_1^-v_1$ \\
            \hline
        \end{tabular}
        \caption{Defining equations for the Weierstrass points on $\tilde{B}$}
        \label{table:WeierstrassPoints}
    \end{table}

    \renewcommand{\arraystretch}{1.7}
	\begin{table}
		\begin{tabular}{|c|rl|}
			\hline
			$\BldP(E_1)$ & 
				$\left[1 - \sqrt2 - i \sqrt{-2 + 2\sqrt2} : 1\right]$, &
				$\left[(i - 1 + i\sqrt2) \sqrt{-2 + 2\sqrt2} : 2\sqrt2\right]$\\[5pt]
				\hline
			$\BldP(E_2)$ & 
				$\left[-1 + \sqrt2 + i \sqrt{-2 + 2\sqrt2} : 1\right]$, &
				$\left[(1 - i - i\sqrt2) \sqrt{-2 + 2\sqrt2} :2\sqrt2\right]$
				\\[5pt]
				\hline
			$\BldP(E_3)$ & 
				$\left[1 - \sqrt2 + i \sqrt{-2 + 2\sqrt2} : 1\right]$, &
				$\left[-(i + 1 + i\sqrt2) \sqrt{-2 + 2\sqrt2} : 2\sqrt2\right]$
				\\[5pt]
				\hline
			$\BldP(E_4)$ & 
				$\left[-1 + \sqrt2 - i \sqrt{-2 + 2\sqrt2} :1\right]$, &
				$\left[(i + 1 + i\sqrt2) \sqrt{-2 + 2\sqrt2} :2\sqrt2\right]$
				\\[5pt]
				\hline
			
		\end{tabular}
		\caption{Defining equation for the exceptional curves on $\dP$}
		\label{table:SingularPoints}
	\end{table}

     \renewcommand{\arraystretch}{1}
\begin{table}
	\begin{tabular}{|c|c|}
		\hline
		$2\pi^*E_1$ & 
			$F_1 + 2G_1 - F_2 + F_3 - F_{10} - F_{12}$\\
		\hline
		$2\pi^*E_2$ & 
			$-F_1 - F_2 + F_3 + F_{10} + F_{12}$\\
		\hline
		$2\pi^*E_3$ & 
			$F_1 + 2G_1 - F_2 - F_3 - F_{10} + F_{12}$\\
		\hline
		$2\pi^*E_4$ & 
			$F_1 + 2G_1 - F_2 - F_3 + F_{10} - F_{12}$\\
		\hline
	\end{tabular}
	\caption{Pullbacks of exceptional curves in terms of $F_i, G_i$}
	\label{table:ExceptionalCurves}
\end{table}

	\begin{landscape}
	
	    \renewcommand{\arraystretch}{1.1}
	\begin{table}[!p]
		\begin{tabular}{|c|c|rl||c|c|rl|}
			\hline
			$\Yabcbar$ & $\Xabcbar$ & \multicolumn{2}{|c||}{Defining equations} &
			$\Yabcbar$ & $\Xabcbar$ & \multicolumn{2}{|c|}{Defining equations} \\
			\hline
			$F_1$ & $C_1$ & $w_0 - \sqrt{5}w_1,$ & $v_0 + 2v_1$&
			$F_4$ & $C_4$ & $\sqrt{c}w_0 - \sqrt5w_2,$ & $
							 10av_0 - (c + \sqrt{c^2-100ab})v_1$\\
            & $\Ctilde_1$ & $w_0 + \sqrt{5}w_1,$ & $ v_0 + v_1$&
            & $\Ctilde_4$ & $\sqrt{c}w_0 + \sqrt5w_2,$ & $
                             10av_0 - (c - \sqrt{c^2-100ab})v_1$\\
            $G_1$ & $D_1$ & $w_0 - \sqrt{5}w_1,$ & $ v_0 + v_1$ &
            $G_4$ & $D_4$ & $\sqrt{c}w_0 - \sqrt5w_2,$ & $
                             10av_0 - (c - \sqrt{c^2-100ab})v_1$\\
            & $\Dtilde_1$ & $w_0 + \sqrt{5}w_1,$ & $ v_0 + 2v_1$&
            & $\Dtilde_4$ & $\sqrt{c}w_0 + \sqrt5w_2,$ & $
                             10av_0 - (c + \sqrt{c^2-100ab})v_1$\\
            \hline
            $F_2$ & $C_2$ & $\sqrt{-2 + 2\sqrt{2}}w_0 - \sqrt5w_1,$ & $
                             v_0 + \sqrt{2}v_1 + \sqrt{5}(1 - \sqrt{2})v_2$&
            $F_5$ & $C_5$ & $i\sqrt2\sqrt[4]{ab}w_0 - w_2,$ &
                        $\sqrt{a}v_0 + \sqrt{b}v_1 + \sqrt{-c - 10\sqrt{ab}}v_2$\\
            & $\Ctilde_2$ & $\sqrt{-2 + 2\sqrt{2}}w_0 + \sqrt5w_1,$ & $
                             v_0 + \sqrt{2}v_1 - \sqrt{5}(1 - \sqrt{2})v_2$ &
            & $\Ctilde_5$ & $i\sqrt2\sqrt[4]{ab}w_0 + w_2,$ &
                        $\sqrt{a}v_0 + \sqrt{b}v_1 - \sqrt{-c - 10\sqrt{ab}}v_2$\\

            $G_2$ & $D_2$ & $\sqrt{-2 + 2\sqrt{2}}w_0 - \sqrt5w_1,$ & $
                             v_0 + \sqrt{2}v_1 - \sqrt{5}(1 - \sqrt{2})v_2$&
            $G_5$ & $D_5$ & $i\sqrt2\sqrt[4]{ab}w_0 - w_2,$ &
                         $\sqrt{a}v_0 + \sqrt{b}v_1 - \sqrt{-c - 10\sqrt{ab}}v_2$\\

            & $\Dtilde_2$ & $\sqrt{-2 + 2\sqrt{2}}w_0 + \sqrt5w_1,$ & $
                             v_0 + \sqrt{2}v_1 + \sqrt{5}(1 - \sqrt{2})v_2$&
            & $\Dtilde_5$ & $i\sqrt2\sqrt[4]{ab}w_0 + w_2,$ &
                        $\sqrt{a}v_0 + \sqrt{b}v_1 + \sqrt{-c - 10\sqrt{ab}}v_2$\\

            \hline
            $F_3$ & $C_3$ & $\sqrt{-2 - 2\sqrt{2}}w_0 - \sqrt5w_1,$ & $
                             v_0 - \sqrt2v_1 + \sqrt5(1 + \sqrt2)v_2$&
            $F_6$ & $C_6$ & $\sqrt2\sqrt[4]{ab}w_0 + w_2,$ &
                        $\sqrt{a}v_0 - \sqrt{b}v_1 + \sqrt{-c + 10\sqrt{ab}}v_2$\\
            & $\Ctilde_3$ & $\sqrt{-2 - 2\sqrt{2}}w_0 + \sqrt5w_1,$ & $
                             v_0 - \sqrt2v_1 - \sqrt5(1 + \sqrt2)v_2$&
            & $\Ctilde_6$ & $\sqrt2\sqrt[4]{ab}w_0 - w_2,$ &
                        $\sqrt{a}v_0 - \sqrt{b}v_1 - \sqrt{-c + 10\sqrt{ab}}v_2$\\
            $G_3$ & $D_3$ & $\sqrt{-2 - 2\sqrt{2}}w_0 - \sqrt5w_1,$ & $
                             v_0 - \sqrt2v_1 - \sqrt5(1 + \sqrt2)v_2$ &
            $G_6$ & $D_6$ & $\sqrt2\sqrt[4]{ab}w_0 + w_2,$ &
                        $\sqrt{a}v_0 - \sqrt{b}v_1 - \sqrt{-c + 10\sqrt{ab}}v_2$\\
            & $\Dtilde_3$ & $\sqrt{-2 - 2\sqrt{2}}w_0 + \sqrt5w_1,$ & $
                             v_0 - \sqrt2v_1 + \sqrt5(1 + \sqrt2)v_2$ &
            & $\Dtilde_6$ & $\sqrt2\sqrt[4]{ab}w_0 - w_2,$ &
                        $\sqrt{a}v_0 - \sqrt{b}v_1 + \sqrt{-c + 10\sqrt{ab}}v_2$\\
			\hline
			\end{tabular}
			\label{table:curves2}
		\end{table}
		\begin{table}[!p]
			\begin{tabular}{|c|c|rl||c|c|rl|}
				\hline
				$\Yabcbar$ & $\Xabcbar$ & \multicolumn{2}{|c||}{Defining equations} &
				$\Yabcbar$ & $\Xabcbar$ & \multicolumn{2}{|c|}{Defining equations} \\
				\hline

			\hline
			$F_7$ & $C_7$ & $\sqrt{c}w_1 - w_2,$ & $(5a + c)v_0 + (c + \gamma_0)v_1$ &
			$F_{10}$ && $w_0 - \sqrt5v_2,$ & $v_0$\\
			& $\Ctilde_7$ & $\sqrt{c}w_1 + w_2,$ & $(5a + c)v_0 + (c - \gamma_0)v_1$ &
			$G_{10}$ && $w_0 - \sqrt5v_2,$ & $v_1$\\\cline{5-8}
			$G_7$ & $D_7$ & $\sqrt{c}w_1 - w_2,$ & $(5a + c)v_0 + (c - \gamma_0)v_1$ &
			$F_{11}$ && $w_2 - \sqrt{c}v_2,$ & $\sqrt{a}v_0 + i\sqrt{b}v_1$\\
			& $\Dtilde_7$ & $\sqrt{c}w_1 + w_2,$ & $(5a + c)v_0 + (c + \gamma_0)v_1$ &
			$G_{11}$ && $w_2 - \sqrt{c}v_2,$ & $\sqrt{a}v_0 - i\sqrt{b}v_1$\\
			\hline
			$F_8$ & $C_8$ & $\theta_2^+ w_1 - w_2,$
			&$v_0 + (2a + b + \theta_0)/(b + \theta_0)v_1 - \theta_1^+/(2a)v_2$ &
			$F_{12}$ && $w_1 - v_2,$ & $v_0 + (1 - i)v_1$\\
			& $\Ctilde_8$ & $\theta_2^+ w_1 + w_2,$
			&$v_0 + (2a + b + \theta_0)/(b + \theta_0)v_1 + \theta_1^+/(2a)v_2$ &
			$G_{12}$ && $w_1 - v_2,$ & $v_0 + (1 + i)v_1$\\\cline{5-8}
			$G_8$ & $D_8$ & $\theta_2^+ w_1 - w_2,$
			&$v_0 + (2a + b + \theta_0)/(b + \theta_0)v_1 + \theta_1^+/(2a)v_2$ &
			$F_{13}$ && $\xi_2^+w_0 - \xi_1^+w_1,$&
				$(\xi_0 + 2\xi_0')v_0 + (2\xi_0 + 2\xi_0')v_1 - w_2$\\
			& $\Dtilde_8$ & $\theta_2^+ w_1 + w_2,$
			&$v_0 + (2a + b + \theta_0)/(b + \theta_0)v_1 - \theta_1^+/(2a)v_2$ &
			$G_{13}$ && $\xi_2^+w_0 - \xi_1^+w_1,$&
				$(\xi_0 + 2\xi_0')v_0 + (2\xi_0 + 2\xi_0')v_1 + w_2$\\
			\hline
			$F_9$ & $C_9$ & $\theta_2^- w_1 - w_2,$
			&$v_0 + (2a + b - \theta_0)/(b - \theta_0)v_1 - \theta_1^-/(2a)v_2$ &
			$F_{14}$ && $\xi_2^-w_0 - \xi_1^-w_1,$&
				$(\xi_0 - 2\xi_0')v_0 + (2\xi_0 - 2\xi_0')v_1 - w_2$\\
			& $\Ctilde_9$ & $ \theta_2^- w_1 + w_2,$
			&$v_0 + (2a + b - \theta_0)/(b - \theta_0)v_1 + \theta_1^-/(2a)v_2$ &
			$G_{14}$ && $\xi_2^-w_0 - \xi_1^-w_1,$&
				$(\xi_0 - 2\xi_0')v_0 + (2\xi_0 - 2\xi_0')v_1 + w_2$\\
				\cline{5-8}
			$G_9$ & $D_9$ & $\theta_2^- w_1 - w_2,$
			&$v_0 + (2a + b - \theta_0)/(b - \theta_0)v_1 + \theta_1^-/(2a)v_2$ &&&&\\
			& $\Dtilde_9$ & $\theta_2^- w_1 + w_2,$
			&$v_0 + (2a + b - \theta_0)/(b - \theta_0)v_1 - \theta_1^-/(2a)v_2$ &&&&\\
			\hline
		\end{tabular}
		\caption{Defining equations of a curve representing a divisor class.}\label{table:curves}\label{table:Picard}
	\end{table}


	\begin{table}
		\begin{tabular}{|l|l|l|l||l|l|l|l|}
			\hline
			Action on & Action on & Action on & Action on &
			Action on & Action on & Action on & Action on \\
			splitting field & $\Pic\Xabcbar$ & 
			$\Pic \Yabcbar$ & Weierstrass  &
			splitting field & $\Pic\Xabcbar$ & 
			$\Pic \Yabcbar$ & Weierstrass \\
			 &  & &  points &
			 &  &  &  points\\
			\hline
			$\sqrt5\mapsto-\sqrt5$ 
				& $C_1 \mapsto D_1$ 
				& $F_1 \leftrightarrow G_1$
				& &
			$i\mapsto -i$ 
				& $C_3 \mapsto \Dtilde_3$ 
				& $F_3 \mapsto G_3$
				& \\
				& $D_1 \leftrightarrow \Ctilde_1$ 
				& & &
				& $C_5 \mapsto \Dtilde_5$ 
				& $F_5 \mapsto G_5$
				& \\
				& $C_2 \mapsto \Ctilde_2$ 
				& & & 
				& & $F_{11} \mapsto G_{11}$ 
				&\\
				& $C_3 \mapsto \Ctilde_3$ 
				& & & 
				& & $F_{12} \mapsto G_{12}$ 
				& \\
			\cline{5-8}
				& $C_4 \mapsto \Dtilde_4$ 
				& $F_4 \mapsto G_4$ 
				& &
			$\sqrt{-2 + 2\sqrt{2}}$
				& $C_2 \mapsto \Dtilde_2$
				& $F_2 \mapsto G_2$
				& \\
				& & $F_{10}\mapsto G_{10}$
				& &
			$ \mapsto -\sqrt{-2 + 2\sqrt{2}}$
				& $C_3 \mapsto\Dtilde_3$
				& $F_3 \mapsto G_3$
				&\\
			\hline
			$\sqrt{2}\mapsto -\sqrt{2}$
				& $C_2 \leftrightarrow C_3$
				& $F_2 \leftrightarrow F_3$
				& &
			$\sqrt{c}\mapsto-\sqrt{c}$
				& $C_4\mapsto \Dtilde_4$
				& $F_4 \mapsto G_4$
				& \\
			$\sqrt{-2 + 2\sqrt{2}}$
				& $C_5 \mapsto \Dtilde_5$
				& $F_5\mapsto G_5$
				& &
				& $C_7\mapsto\Dtilde_7$
				& $F_7 \mapsto G_7$
				& \\
			$\mapsto \sqrt{-2 - 2\sqrt2}$
				& $C_6\mapsto \Dtilde_6$
				& $F_6\mapsto G_6$
				& &
				& & $F_{11}\mapsto G_{11}$
				& \\
			\hline
			$\gamma_0\mapsto -\gamma_0$
				& $C_7\mapsto D_7$
				& $F_7 \mapsto G_7$
				& $Q_1 \leftrightarrow Q_3$&
			$\eta_0\mapsto - \eta_0$
				& $C_4 \mapsto D_4$
				& $F_4 \mapsto G_4$
				& $P_1\leftrightarrow P_3$\\
				& $C_9 \mapsto D_9$
				& $F_9 \mapsto G_9$
				& $Q_2 \leftrightarrow Q_4$&
				& $C_6\mapsto D_6$
				& $F_6\mapsto G_6$
				& $P_2\leftrightarrow P_4$\\
				& & $F_{14}\mapsto G_{14}$
				& &
				& & $F_{14}\mapsto G_{14}$
				& \\
			\hline
			$\sqrt[4]{ab}\mapsto i\sqrt[4]{ab}$
				& $C_5\mapsto C_6$
				& $F_5 \mapsto G_6$
				& $P_3\leftrightarrow P_4$&
			$\sqrt{a} \mapsto - \sqrt{a}$
				& $C_5\mapsto D_5$
				& $F_5\mapsto G_5$
				& $P_1\leftrightarrow P_2$\\
			$\sqrt{-c - 10\sqrt{ab}}$
				& $C_6\mapsto \Dtilde_5$
				& $F_6 \mapsto G_5$
				& & 
				& $C_6 \mapsto D_6$
				& $F_6\mapsto G_6$
				& $P_3\leftrightarrow P_4$\\
			\cline{5-8}
			$\mapsto \sqrt{-c+10\sqrt{ab}}$
				& $C_9 \mapsto \Dtilde_9$
				& $F_9 \mapsto G_9$
				&&
			$\sqrt{-c - 10\sqrt{ab}}$
				& $C_5 \mapsto D_5$
				& $F_5\mapsto G_5$
				& \\
				& & $F_{11} \mapsto G_{11}$
				& & 
			$\mapsto - \sqrt{-c- 10\sqrt{ab}}$
				& $C_6\mapsto D_6$
				& $F_6\mapsto G_6$
				& \\
			\hline
			$\xi_0\mapsto -\xi_0$
				& & $F_{13}\mapsto G_{14}$
				& &
			$\theta_0\mapsto-\theta_0$
				& $C_8\leftrightarrow C_9$
				& $F_8 \leftrightarrow F_9$
				& $Q_3\leftrightarrow Q_4$\\
			$\xi_i^+\mapsto \xi_i^-$
				& & $F_{14}\mapsto G_{13}$
				& &
			$\theta_i^+\mapsto\theta_i^-$
				& & & \\
			\hline
			$\xi_0'\mapsto -\xi_0'$
				& & $F_{13}\leftrightarrow F_{14}$
				& &
			$\theta_1^+\mapsto-\theta_1^+$
				& $C_8\mapsto D_8$
				& $F_8 \mapsto G_8$
				& $Q_1\leftrightarrow Q_2$\\
			$\xi_i^+\mapsto \xi_i^-$
				& & & &
				& $C_9\mapsto D_9$
				& $F_9\mapsto G_9$
				& $Q_3\leftrightarrow Q_4$\\
			\hline
			$\xi_1^+\mapsto -\xi_1^+$
				& & $F_{13}\mapsto G_{13}$
				& & 
			$\theta_2^+ \mapsto -\theta_2^+$
				& $C_8 \mapsto \Dtilde_8$
				& $F_8\mapsto G_8$
				& \\
				& & $F_{14}\mapsto G_{14}$
				& & 
				& $C_9 \mapsto \Dtilde_9$
				& $F_9 \mapsto G_9$
				& \\
			\hline
			$\xi_2^+\mapsto -\xi_2^+$
				& & $F_{13}\mapsto G_{13}$
				& \\
				& & $F_{14}\mapsto G_{14}$
				& \\
			\cline{1-4}

		\end{tabular}
		\caption{The Galois action on the fibers of the genus 1 fibrations and the Weierstrass points.\newline
		The action on the splitting field is described by the action on the generators listed in~\ref{eqn:gens1},~\ref{eqn:gens2},~\ref{eqn:gens3},~\ref{eqn:gens4},~\ref{eqn:gens5}.  If a generator of $K$ is not listed, then we assume that it is fixed.  We use the same convention for the curve classes and the Weierstrass points.}\label{table:galois}\label{table:Galois}
	\end{table}
	
	\end{landscape}   

\end{document}

%% file: Introduction.tex
\ifx
    \thepage\undefined\def\jobname{BBMPV_EnriquesTranscendental}
    \input{BBMPV_EnriquesTranscendental}
\fi

\section{Introduction}\label{sec:Introduction}

	Given a smooth, projective, geometrically integral variety $X$ over a global field $k$, one may ask whether $X$ has a $k$-rational point, that is, whether $X(k) \neq \emptyset$.  Since $k$ embeds into each of its completions, a necessary condition for $X(k) \neq \emptyset$ is that $X(\A_k)\neq\emptyset$.  However, this condition is often not sufficient; varieties $X$ with $X(\A_k)\neq\emptyset$ and $X(k) = \emptyset$ exist, and these are said to \defi{fail the Hasse principle}.

    In 1970, Manin~\cite{Manin-ICM} significantly advanced the study of failures of the Hasse principle {by use of} the Brauer group and class field theory.  More precisely, he defined a subset $X(\A_k)^{\Br}$ of $X(\A_k)$, now known as the \defi{Brauer-Manin set}, with the property that
	\[
		X(k) \subset X(\A_k)^{\Br} \subset X(\A_k).
	\]
	Thus, we may think of an empty Brauer-Manin set as an \defi{obstruction} to the existence of rational points. At the time of its introduction, this Brauer-Manin obstruction explained all known failures of the Hasse principle. 
	
	In 1999, Skorobogatov~\cite{Sko-etaleBrauer} defined a refinement of the Brauer-Manin set, the \defi{\'etale-Brauer set} $X(\A_k)^{\et,\Br}$, which still contains $X(k)$. He proved that this new obstruction can be stronger than the Brauer-Manin obstruction, by constructing a bielliptic surface $X/\Q$ such that $X(\A_\Q)^{\et,\Br} = \emptyset$ and $X(\A_\Q)^{\Br} \neq\emptyset$.

	Bielliptic surfaces have a number of geometric properties in common with Enriques surfaces: both have Kodaira dimension $0$ and  nontrivial \'etale covers.  This raises the natural question of whether the \'etale-Brauer obstruction is stronger than the Brauer-Manin obstruction \emph{for Enriques surfaces}. 	Harari and Skorobogatov took up this question in 2005; they constructed an Enriques surface $X/\Q$ whose \'etale-Brauer set was strictly smaller than the Brauer-Manin set~\cite{HS-Enriques}, thereby showing that the Brauer-Manin obstruction is insufficient to explain all failures of {weak approximation}\footnote{A smooth projective variety $X$ \defi{satisfies weak approximation} if $X(k)$ is dense in $X(\A_k)$ in the adelic topology.} on Enriques surfaces. Their surface, however, had a $\Q$-rational point, so it did not fail the Hasse principle.

	The main result of this paper is the analogue of Harari and Skorobogatov's result for the Hasse principle.  Precisely, we prove:
	\begin{theorem}\label{thm:Main}
		The Brauer-Manin obstruction is insufficient to explain all failures of the Hasse principle on Enriques surfaces.
	\end{theorem}

	This theorem builds on work of V\'arilly-Alvarado and the last author.  To explain the connection, we must first provide more information about the Brauer group.  For any variety $X/k$, we have the following filtration of the Brauer group:
	\[
		\Br_0 X := \im (\Br k \to \Br X) \subset \Br_1 X := \ker(\Br X \to \Br X_{k^{\sep}}) \subset \Br X = \HH^2_{\et}(X,\G_m).
	\]
	Elements in $\Br_0 X$ are said to be \defi{constant}, elements in $\Br_1 X$ are said to be \defi{algebraic}, and the remaining elements are said to be \defi{transcendental}.  
	The Brauer-Manin set $X(\A_k)^{\Br}$ depends only on the quotient $\Br X/\Br_0 X$ (this follows from the fundamental exact sequence of class field theory, see~\cite[\S5.2]{Skorobogatov-Torsors} for more details).  As transcendental Brauer elements have historically been difficult to study, one sometimes instead considers the (possibly larger) \defi{algebraic Brauer-Manin set} $X(\A_k)^{\Br_1}$, defined in terms of the subquotient $\Br_1 X/\Br_0 X$.

	We now recall the main result of~\cite{VAV-Enriques}.
	\begin{theorem*}[{\cite[Thm. 1.1]{VAV-Enriques}}]
		There exists an Enriques surface $X/\Q$ such that 
		\[
			X(\A_\Q)^{\et,\Br} = \emptyset \quad\textup{and}\quad
			X(\A_\Q)^{\Br_1}\neq\emptyset.
		\]
	\end{theorem*}
	The proof of~\cite[Thm. 1.1]{VAV-Enriques} is constructive.  Precisely, for any $\mathbf{a} = (a,b,c)\in \Z^3$ with
	\begin{equation}\label{eq:Nonsingular}
		abc(5a + 5b + c)(20a + 5b + 2c)(4a^2 + b^2)(c^2 -100ab)(c^2 + 5bc + 10ac + 25ab)\neq 0,
	\end{equation} 
	the authors consider $\Yabc\subset\PP^5$, the smooth degree-$8$ K3 surface given by 
	\begin{align*}
	v_0v_1 + 5v_2^2 &= w_0^2,\\
	(v_0 + v_1)(v_0+2v_1) &= w_0^2 - 5w_1^2,\\
	av_0^2 +bv_1^2 + cv_2^2 &= w_2^2.
	\end{align*}
	The involution $\sigma\colon \PP^5\to \PP^5, (v_0:v_1:v_2:w_0:w_1:w_2)\mapsto(-v_0:-v_1:-v_2:w_0:w_1:w_2)$ has no fixed points on $\Yabc$ so the quotient $\Xabc := \Yabc/\sigma$ is an Enriques surface.

\begin{theorem*}[{\cite[Theorem 1.2]{VAV-Enriques}}]
Let $\mathbf a = (a, b, c) \in \Z^3_{>0}$
satisfy the following conditions:
\begin{enumerate}
\item for all prime numbers $p \mid (5a + 5b +c)$, 5 is not a square modulo $p$,
\item for all prime numbers $p \mid (20a + 5b +2c)$, 10 is not a square modulo $p$,
\item the quadratic form $av_0^2 +bv_1^2 +cv_2^2 + w_2^2$ is anisotropic over $\Q_3$,
\item the integer $-bc$ is not a square modulo $5$,
\item the triplet $(a, b, c)$ is congruent to $(5, 6, 6)$ modulo $7$,
\item the triplet $(a, b, c)$ is congruent to $(1, 1, 2)$ modulo $11$,
\item $Y_{\mathbf a}(\A_\Q) \neq \emptyset$,
\item the triplet $(a, b, c)$ is \emph{Galois general} (meaning that a certain number field defined in terms of
$a$, $b$, $c$ is as large as possible).
\end{enumerate} 
Then
\[
X_{\mathbf a}(\A_\Q)^{\et,\Br} = \emptyset  \quad \text{ and } \quad X_{\mathbf a}(\A_\Q)^{\Br_1}\neq \emptyset.
\]
\end{theorem*}
V\'arilly-Alvarado and the last author deduce~\cite[Thm. 1.1]{VAV-Enriques} from~\cite[Thm. 1.2]{VAV-Enriques} by showing that the triplet $\mathbf a = (12, 111, 13)$ satisfies conditions (1)--(8).  Henceforth, when we refer to ``conditions'' by number, we mean the conditions given in the theorem above.

In~\cite{VAV-Enriques}, the authors have left open the question of a transcendental obstruction to the Hasse principle for the surfaces $X_{\mathbf a}$, due to the ``difficulty [\dots] in finding an explicit representative for [the nontrivial] Brauer class of $[\Xabcbar]$.''  Recent work of Creutz and the last author~\cites{CV-BrauerSurfaces, CV-BrauerCurves}, and Ingalls, Obus, Ozman, and the last author~\cite{IOOV-UnramifiedBrauer} makes this problem more tractable.  Building on techniques from~\cites{CV-BrauerSurfaces, CV-BrauerCurves, IOOV-UnramifiedBrauer}, we prove:

	\begin{theorem}\label{thm:MainPrecise}
		If $\mathbf{a} = (a,b,c)\in \Z^3_{>0}$ satisfies conditions $(5)$, $(6)$, and $(8)$ then $\Br \Xabc = \Br_1\Xabc$.  In particular, if $\mathbf{a}$ satisfies conditions $(1)$--$(8)$, then
		\[
			X_{\mathbf a}(\A_{\Q})^{\et,\Br} = \emptyset \quad\textup{and}\quad
			X_{\mathbf a}(\A_\Q)^{\Br} \neq\emptyset.
		\]
	\end{theorem}

	\subsection*{Strategy and outline}
		Theorem~\ref{thm:Main} and the second statement of Theorem~\ref{thm:MainPrecise} both follow immediately from the first statement of Theorem~\ref{thm:MainPrecise} and~\cite{VAV-Enriques}, since the triplet $\mathbf a = (12, 111, 13)$ satisfies conditions (1)--(8)~\cite[Lemma 6.1 and Proof of Theorem 1.1]{VAV-Enriques}.  Thus, we reduce to proving the first statement of Theorem~\ref{thm:MainPrecise}.  

		For any variety $X$ over a field $k$, the quotient $\Br X/\Br_1 X$ injects into $(\Br X_{k^{\sep}})^{\Gal(k^{\sep}/k)}$.  In Skorobogatov's pioneering paper, his construction $X/\Q$ had the additional property that $(\Br X_{\Qbar})^{\Gal(\Qbar/\Q)} = 0$, so $\Br X = \Br_1 X$.  Unfortunately, this strategy cannot be applied to an Enriques surface $X$, as $\Br X_{k^{\sep}} \isom \Z/2\Z$~\cite[p. 3223]{HS-Enriques} and hence the unique nontrivial element is always fixed by the Galois action.

		Instead, we will find a Galois extension $K_1/\Q$ and an open set $U'\subset \Xabc$ such that
		\begin{enumerate}
			\item the \defi{geometrically unramified subgroup} $\Br^{\textup{g.unr.}} U'_{K_1}\subset \Br U'_{K_1}$ (i.e., the subgroup of elements in $\Br U'_{K_1}$ which are contained in $\Br \Xabcbar\subset \Br \overline{U'}$ upon base change to $\Qbar$) surjects onto $\Br \Xabcbar$, and
			\item $(\Br^{\textup{g.unr.}} U'_{K_1}/\Br K_1)^{\Gal(K_1/\Q)}$ is contained in $\Br_1 U'_{K_1}/\Br K_1$.  
		\end{enumerate}
		
		{The key step is proving $(2)$ \emph{without} necessarily having central simple $\kk(U'_{K_1})$-algebra representatives for all of the elements of $\Br^{\textup{g.unr.}} U'_{K_1}/\Br K_1$.  Our approach follows the philosophy laid out in~\cites{CV-BrauerSurfaces, CV-BrauerCurves, IOOV-UnramifiedBrauer}: we study geometrically unramified Brauer classes on $U'_{K_1}$ via pullback of ramified Brauer classes on simpler surface $S'$, of which $U'$ is a double cover.  However, in contrast to the work of~\cites{CV-BrauerSurfaces, CV-BrauerCurves, IOOV-UnramifiedBrauer}, we carry this out over a field that is not necessarily separably closed.  In particular, our methods can be carried out over a number field.  As we expect this approach to be of independent interest, we build up some general results in~\S\ref{sec:Pullback} which can be applied to a double cover of a rational ruled surface, assuming mild conditions on the branch locus.  }
		
        \begin{remark}
            For convenience, we carry out the above strategy on the K3 surface $\Yabc$ instead of on the Enriques surface $\Xabc$. We then descend the results to $\Xabc$.
        \end{remark}

		Starting in~\S\ref{sec:Geometry}, we restrict our attention to the specific varieties $\Xabc$ and $\Yabc$. After recalling relevant results from~\cite{VAV-Enriques}, we construct double cover maps $\pi\colon \Yabc\to S$ and $\tilde\pi\colon \Xabc \to \Stilde$, where $S$ and $\Stilde$ are ruled surfaces, and we study the geometry of these morphisms.  These maps allow us to apply the results of~\cite{CV-BrauerSurfaces} to construct, in~\S\ref{sec:Transcendental}, an explicit central simple $\kk(\Xabcbar)$-algebra representative $\calA$ of the nontrivial Brauer class of $\Xabcbar$.  This representative $\calA$ will necessarily be defined over a number field $K_1$, be unramified over an open set $U'_{K_1}$, and be geometrically unramified.  Furthermore, the number field $K_1$ and the open set $U'$ can be explicitly computed from the representative $\calA$.

		Section~\ref{sec:Proof} uses the results from~\S\ref{sec:Pullback} to study the action of $\Gal(\Qbar/K_1)$ on $\Br^{\textup{g.unr.}} U'_{K_1}/\Br K_1$ and hence prove Theorem~\ref{thm:MainPrecise}.  Namely, by repeated application of the commutative diagram in Theorem~\ref{thm:Br1vsBr}, we demonstrate that no $\sigma$-invariant transcendental Brauer class can exist for $\Yabc$.  Indeed, if such a class existed, the explicit central simple algebra from~\S\ref{sec:Transcendental} would relate it to a function $\tilde\ell$ fixed by the Galois action.  However, direct computation (given in the appendix)  shows that $\tilde\ell$ must be moved by some Galois action, providing the required contradiction.

	\subsection*{General notation}

		Throughout, $k$ will be a field with characteristic not equal to $2$, with fixed separable closure $\bar k$. For any $k$-scheme $X$ and field extension $k' / k$, we write $X_{k'}$ for the base change $X \times_{\Spec k} \Spec k'$ and  $\overline{X}$ for the base change $X \times_{\Spec k} \Spec{\kbar}$.  If $X$ is integral, we write $\kk(X)$ for the function field of $X$. We also denote the absolute Galois group by $G_{k} = \Gal(\bar k/k)$. For any $k$-variety $W$, we use the term splitting field (of $W$) to mean the smallest Galois extension over which every geometrically irreducible component of $W$ is defined.

		The Picard group of $X$ is $\Pic X := \Div X / \Princ X$, where $\Div X$ is the group of Weil divisors on $X$ and $\Princ X$ is the group of principal divisors on $X$.    Let $\Pic^0 X$ denote the connected component of the identity in $\Pic X$; then the N\'eron-Severi group of $X$ is $\NS X := \Pic X / \Pic^0 X$.  For a divisor $D \in \Div X$, we write $[D]$ for its equivalence class in $\Pic X$.  When $X$ is a curve, the Jacobian of $X$ satisfies $\Jac X = \Pic^0 X$.

		For a $k$-scheme $Y$, we write $\Br Y$ for the \'etale cohomology group 
		$\Br Y := \HH^2_\text{\'et} (Y,\G_m)$. The geometrically unramified subgroup $\Br^{\textup{g.unr.}} \kk(Y) \subset\Br \kk(Y)$ consists of those Brauer classes which are contained in $\Br \Ybar$ upon base change to $\kbar$. For an open subscheme $U \subset Y$, we have $\Br^{\textup{g.unr.}} U := \Br U \cap \Br^{\textup{g.unr.}} \kk(Y)$. 
         If $A$ is an \'etale $k$-algebra, then we  write $\Br A$ for $\Br (\Spec A)$. Given invertible elements $a$ and $b$ in such an $A$, we define the quaternion algebra $(a, b) := A[i, j]/\langle i^2 = a, j^2 = b, ij = -ji\rangle$.  We will identify the algebra $(a,b)$ with its class in $\Br A$. 

		Now assume that $Y$ is smooth and quasi-projective. Then the following sequence is exact:  
		\begin{equation} \label{eqn: purity}
			0 \rightarrow \Br Y[2] \rightarrow \Br \kk(Y)[2] \xrightarrow{\oplus_y\partial_y} \bigoplus_y \HH^1(\kk(y), \Z/2\Z),
		\end{equation}
		where the sum is taken over the set of  all codimension-$1$ points $y$ on $Y$~\cite[Thm. 6.1]{Gro-DixExp}. As $\Br \kk(Y)[2]$ is generated by quaternion algebras, we will only describe the residue map $\partial_y$ on quaternion algebras:	for any $a,b\in \kk(Y)^{\times}$, we have
			\[
				\partial_y\left((a,b)\right) = (-1)^{v_y(a)v_y(b)}a^{v_y(b)}b^{-v_y(a)} \in \kk(y)^{\times}/\kk(y)^{\times2} \isom \HH^1\left(\kk(y), \Z/2\Z\right),
			\]
		where $v_y$ denotes the valuation corresponding to $y$.

%% file: PullbackBrauer.tex
\ifx
    \thepage\undefined\def\jobname{BBMPV_EnriquesTranscendental}
    \input{BBMPV_EnriquesTranscendental}
\fi

    \section{Brauer classes on double covers arising via pullback}
    \label{sec:CVrecap}\label{sec:Pullback}

        		Let $\pi^0\colon Y^0 \to S^0$ be a double cover of a smooth rational geometrically ruled surface $\varpi\colon S^0\to \PP^1_t$ defined over $k$ and let 
		$B^0 \subset S^0$
 denote the branch locus of $\pi^0$. We assume that $B^0$ is reduced, geometrically irreducible, and has at worst ADE singularities.  The \defi{canonical resolution}~\cite[Thm.~7.2]{BHPVdV} 
$\nu\colon Y\to Y^0$
 of $\pi^0\colon Y^0\to S^0$ has a $2$-to-$1$ $k$-morphism $\pi\colon Y \to S$ to a smooth rational \emph{generically} ruled surface $S$; the branch curve $B\subset S$ of $\pi$ is a smooth proper model of $B^0$.  In summary, we have the following diagram:

\[
\begin{tikzcd}[column sep = large]
Y \ar{r}{\pi} \ar{d}{\nu} & S \ar{d}{\nu_S} &[-1.8 cm]\supset B\\
Y^0 \ar{r}{\pi^0} & S^0 &[-1.8 cm] \supset B^0
\end{tikzcd}
\]
Since $B^0$ is geometrically irreducible,  $\Pic^0 Y$ is trivial by~\cite[Cor. 6.3]{CV-BrauerSurfaces} and so we may conflate $\Pic Y$ and $\NS Y$.

    	The generic fiber of $\varpi\circ\pi^0$ is a double cover $C \to \PP^1_{k(t)}\to \Spec k(t)$. Since $k(t)$ is infinite, by changing coordinates if necessary, we may assume that the double cover is unramified above $\infty\in S^0_{k(t)}$. Then $C$ has a model 
\[y^2 = c' h(x),\]
 for some $c' \in k(t)$ and $h\in k(t)[x]$ square-free, monic, and with $\deg(h) = 2 g(C) +2$, where $g(C)$ denotes the genus of $C$. Note that  $\kk(B) = \kk(B^0) \cong k(t)[\theta]/(h(\theta))$; we write $\alpha$ for the image of $\theta$ in $\kk(B)$.

        As $S^0$ is rational and geometrically ruled, $\Pic S^0 \isom \Z^2$ and is generated by a fibre $S^0_{\infty}$, and a section $\frakS$, which we may take to be the closure of $x = \infty$. Since $\nu\colon S \to S^0$ is a birational map, $\Pic \overline{S}$ is generated by the strict transforms of $\frakS$ and $S^0_\infty$ and the curves $E_1, \dots, E_n$ which are contracted by the map $S\to S^0$. We will often abuse notation and conflate $\frakS$ and $S^0_{\infty}$ with their strict transforms. Let 
        \[
            \calE = \left\{\frakS, S^0_{\infty}, E_1,\dots, E_n\right\}
        \]
         denote the aforementioned set of $n+2$ generators and define 
 \[
			V:= S \setminus ( \calE \cup B) \subset S.
		\]
   Possibly after replacing $k$ with a finite extension, we may assume that all elements of $\calE$ are defined over $k$ and, in particular, that $\Pic S = \Pic \overline{S}$. Since $\nu$ is defined over $k$, we additionally have $\Br S = \Br S^0 = \Br k$. 		  

		For any $\ell\in \kk(B)^{\times}$, we define
		\[
			\calA_{\ell} := \Cor_{\kk(B)(x)/k(t,x)}\left((\ell, x - \alpha)\right)\in\Br\kk(S).
		\]
		We will be particularly concerned with $\calA_{\ell}$ when $\ell$ is contained in the subgroup
		\begin{align*}
			\kk(B)_{\calE}^{\times}  := & \left\{\ell\in \kk(B)^{\times} : \divv(\ell) \in \im (\Z^{\calE} \to \Div(S) \to \Div(B)\otimes \Z/2\Z)\right\}\\
			 =& \left\{\ell\in \kk(B)^{\times} : \nu_*\divv(\ell) \in \langle \frakS, S^0_{\infty}\rangle\subset \Div(B^0)\otimes \Z/2\Z\right\}.
		\end{align*}
		By~\cite[Proof of Thm. 5.2]{CV-BrauerSurfaces}, this subgroup is exactly the set of functions $\ell$ such that $\pi^*\calA_{\ell}$ is geometrically unramified.  Note that $\kk(B)_{\calE}^{\times}$ contains $k^{\times}\kk(B)^{\times2}.$

Let 
\[ 
U := Y \setminus (\pi^{-1}(\calE)) \subset Y.
\]

		The goal of this section is to prove the following two theorems.
		\begin{thm}\label{thm:NumberFieldExactSequence}
			Let $k'$ be any Galois extension of $k$.  Then we have the following exact sequence of $\Gal(k'/k)$-modules:
		    \begin{equation}\label{eq:thm22}
		        0 \to  \frac{\Pic Y_{k'}}{\pi^*\Pic S + 2\Pic Y_{k'}} \stackrel{j}{\To} 
		            \frac{\kk(B_{k'})_{\calE}^{\times}}{k'^{\times}\kk(B_{k'})^{\times2}}
					\stackrel{\beta}{\To}
					\left(\frac{\Br^{\textup{g. unr.}} U_{k'}}{\Br k'}\right)[2],
		    \end{equation}
                where 
 $j$ is as in~\S\ref{sec:DefinitionOfj} and $\beta$ is as in~\S\ref{sec:DefinitionOfBeta}.
                Furthermore,  if $k'$ is separably closed, then the last map surjects onto $\Br Y[2]$.
		\end{thm}
        \begin{theorem}\label{thm:Br1vsBr}
            We retain the notation from Theorem~\ref{thm:NumberFieldExactSequence}.
            If $\Br k' \to \Br \kk(S_{k'})$ is injective and $\Pic \Ubar$ is torsion free, then there is a commutative diagram of $\Gal(k'/k)$-modules with exact rows and columns:
            \[
                \xymatrix{            
                    & \frac{\Pic Y_{k'}}{\pi^*\Pic S + 2\Pic Y_{k'}} \ar[r]^(.6)j \ar@{^{(}->}[d]&
                    j(\Pic Y_{k'})\ar@{^{(}->}[d]\\
                    0\ar[r]&
                    \left(\frac{\Pic \Ybar}{\pi^*\Pic S + 2\Pic \Ybar}\right)^{G_{k'}}
                    \ar[r]^(.6)j \ar@{->>}^{\beta\circ j}[d] & 
                    \frac{\kk(B_{k'})_{\calE}}{k'^{\times}\kk(B_{k'})^{\times2}}
                    \ar^{\beta}[r]\ar^\beta[d]&
                   \frac{\Br^{\textup{g.unr.}} U_{k'}}{\Br_1 U_{k'}} \ar@{=}[d]\\
                    0 \ar[r] & \frac{\Br_1 U_{k'}}{\Br k'} \ar[r] & 
                    \frac{\Br^{\textup{g.unr.}} U_{k'}}{\Br k'} \ar[r] &
                    \frac{\Br^{\textup{g.unr.}} U_{k'}}{\Br_1 U_{k'}} \ar[r] & 0.\\
                }
            \]
        \end{theorem}

        The structure of the section is as follows.  In~\S\ref{sec:Residues}, we prove some preliminary results about the residues of $\calA_{\ell}$; these are used in  section~\S\ref{sec:DefinitionOfBeta} to define the map $\beta$.  Next, in~\S\ref{sec:DefinitionOfj}, we define $j$ and prove that it is injective.  In~\S\ref{sec:KernelPullback}, we characterize the elements of $\Br V$ that pull back to constant algebras under $\pi^*$.  In~\S\ref{sec:ProofOfNumberFieldExactSequence}, we combine the results from the earlier sections to prove Theorem~\ref{thm:NumberFieldExactSequence}, and, finally, in~\S\ref{sec:ProofOfThmBr1vsBr}, we prove Theorem~\ref{thm:Br1vsBr}.

		\subsection{Residues of $\calA_{\ell}$}\label{sec:Residues}

		    In order to define the homomorphism $\beta$, we will need to know the certain properties about the residues of $\calA_{\ell}$ at various divisors of $S^0$. We first compute residues associated to horizontal divisors.

    		\begin{lemma}\label{lem: HorizResidues}
    			Let $\ell\in \kk(B)^{\times}$, and let $F$ be an irreducible horizontal curve in $S^0$, i.e., a curve that maps dominantly onto $\PP^1_t$.
    			\begin{enumerate}
    				\item If $F \ne B, \frakS$, then $\partial_F(\calA_\ell) = 1 \in \kk(F)^\times/ \kk(F)^{\times2}$. 
    				\item $\partial_{B}(\calA_{\ell}) = [\ell]\in \kk(B)^{\times}/\kk(B)^{\times2}$.
    			\end{enumerate}
    		\end{lemma}

    		\begin{proof}
                The arguments in this proof follow those in~\cite[Proofs of Thm. 1.1 and Prop. 2.3]{CV-BrauerCurves}; as the situation is not identical, we restate the arguments here for the reader's convenience.  

                Let $v$ be the valuation on $\kk(S^0)$ associated to $F$.  By~\cite[Lemma 2.1]{CV-BrauerCurves}, we have
                \begin{equation}\label{eq:Residues}
                    \partial_F(\calA_{\ell}) = \prod_{w|v}\Norm_{\kk(w)/\kk(v)} \left((-1)^{w(\ell)w(x - \alpha)}\ell^{w(x-\alpha)}(x - \alpha)^{-w(\ell)}\right),
                \end{equation}
                where $w$ runs over all valuations on $\kk(B\times_{\PP^1_t}S^0)$ extending $v$.  As $F$ is a horizontal divisor, $v|_{k(t)}$ is trivial and hence $w|_{\kk(B)}$ is trivial for all $w|v$.  Therefore,~\eqref{eq:Residues} simplifies to $\prod_{w|v}\Norm_{\kk(w)/\kk(v)} (\ell^{w(x-\alpha)})$.

                By definition of $\alpha$, $\Norm_{\kk(B)(x)/k(t)(x)}(x - \alpha)= h(x)$.  Thus, $w(x - \alpha) = 0$ for all $w|v$ if $v(h(x)) = 0$, or equivalently, if $F \neq B, \frakS$.  This completes the proof of $(1)$.

                Now assume that $F = B$.  We know that $h(x)$ factors as $(x-\alpha)h_1(x)$ over $\kk(B)(x)$, where $h_1 \in k(t)[x]$ is possibly reducible. Hence, $x-\alpha$ determines a valuation $w_{x-\alpha}$ on $\kk(B)(x)$ lying over $v$; similarly, the other irreducible factors of $h_1$ also determine valuations lying over $v$. Notice that since $h(x)$ is separable (as $B$ is reduced), we have that $h_1(\alpha) \neq 0$, and hence that $w(x -\alpha) = 0$ for any valuation $w$ over $v$ corresponding to the irreducible factors of $h_1(x)$. Thus,~\eqref{eq:Residues} simplifies to 
                \[
                \prod_{w|v}\Norm_{\kk(w)/\kk(v)} (\ell^{w(x-\alpha)}) = 
                \Norm_{\kk(w_{x-\alpha})/\kk(v)} (\ell) = \ell,
                \]
                as required.
                	\end{proof}

            Now we compute the residues associated to vertical divisors. 
            
    		\begin{lemma}\label{lem: VertResidues}
    			Let $\ell\in \kk(B)^{\times}_{\calE}$, $t_0 \in \A^1_t \subset \PP^1_t$, and $F = S^0_{t_0}$. Then,
    			\[
    				\partial_{F}(\calA_\ell) \in \im\left(\frac{\kk(t_0)^\times}{\kk(t_0)^{\times 2}}  \to  \frac{\kk(F)^\times}{\kk(F)^{\times 2}} \right).
    			\]
    		\end{lemma}
			\begin{remark}
				If $k$ is separably closed, then $\kk(t_0)^{\times2} = \kk(t_0)^{\times}$ and the result follows from~\cite[Prop. 3.1]{CV-BrauerCurves}.
			\end{remark}
    		\begin{proof}
    			It suffices to show that $\partial_{F}(\calA_\ell) \in \kk(F)^{\times 2} \kk(t_0)^\times$. We repeat~\cite[{Proof of Prop. 3.1}]{CV-BrauerSurfaces} while keeping track of scalars to accommodate the fact that  $k$ is not necessarily separably closed.

    			By~\cite[Lemma 2.1]{CV-BrauerCurves}, we have

    			\begin{equation} \label{eq:resatF}
    				\partial_{F}(\calA_\ell) =
    				\prod_{\substack{
    					F'\in S^0\times_{\PP^1_t}B\\
    					F' \mapsto F \textup{ dominantly}
    					}}
    				\Norm_{\kk(F')/\kk(F)} ((-1)^{w'(x-\alpha) w'(\ell)}\ell^{w'(x-\alpha)}(x-\alpha)^{-w'(\ell)}),
    			\end{equation}
    			where $F'$ is an irreducible curve and $w'$ denotes the valuation associated to $F'$.  The surface $S^0\times_{\PP^1_t} B$ is a geometrically ruled surface over $B$, so the irreducible curves $F'$ are in one-to-one correspondence with points $b'\in B$ mapping to $t_0$.  Furthermore, $\kk(F') = \kk(b')(x)$ and $\kk(F) = \kk(t_0)(x)$, so $\Norm_{\kk(F')/\kk(F)}$ is induced from $\Norm_{\kk(b')/\kk(t_0)}.$  Thus, we may rewrite~\eqref{eq:resatF} as
                \begin{equation}\label{eq:resatFv2}
                    \partial_{F}(\calA_\ell) =
    				\prod_{b'\in B, b'\mapsto t_0}
    				\Norm_{\kk(b')/\kk(t_0)} ((-1)^{w'(x-\alpha) w'(\ell)}\ell^{w'(x-\alpha)}(x-\alpha)^{-w'(\ell)}).
                \end{equation}

                By~\cite[Lemma 3.3]{CV-BrauerSurfaces}, there exists an open set $W\subset \A^1$ containing $t_0$ and constants $d\in\kk(t)^{\times}$, $e\in\kk(t)$ such that
                \[
                    S^0_W \to \PP^1_k\times W, \quad s\mapsto(dx(s) + e, \varpi(s))
                \]
                is an isomorphism.  In particular, $dx + e$ is a horizontal function on $S^0_W$.  Consider the following equality:
                \[
                \Cor_{\kk(B)(x)/\kk(S^0)}\left((dx + e - (d\alpha + e), \ell)\right)
                = \calA_{\ell} + \Cor_{\kk(B)(x)/\kk(S^0)}( (d, \ell)) 
                = \calA_{\ell} + (d, \Norm(\ell)).
                \]
                Since $(d,\Norm(\ell))\in \varpi^*\Br k(t)$, we have
                \[
                \partial_F(\calA_{\ell}) \in \partial_F\left(\Cor_{\kk(B)(x)/\kk(S^0)}\left((dx + e - (d\alpha + e), \ell)\right)\right)\kk(t_0)^{\times}.
                \]
                 Thus, we may assume that $x$ is a horizontal function, in particular, that $x$ has  no zeros or poles along $F$, and that it restricts to a non-constant function along $F$.
    			It is then immediate that the function $w'(x - \alpha)\leq 0$, and that the inequality is strict if and only if $w'(\alpha)<0$, which in turn happens if and only if $b'$ lies over $B^0_{t_0}\cap \frakS$.
            
                We first consider the factor of~\eqref{eq:resatFv2} that corresponds to points that do not lie over $B^0_{t_0}\cap \frakS$.  If $b'$ does not lie over $B^0_{t_0}\cap \frakS$, then (as stated above) $w'(x-\alpha) = 0$, where $w'$ denotes the valuation associated to $b'$. Therefore, the corresponding factor of~\eqref{eq:resatFv2} simplifies to 
                \[
                    \prod_{b'\in B\setminus \nu^{-1}(B^0\cap\frakS), b'\mapsto t_0}\Norm_{\kk(b')/\kk(t_0)}\left((x - \alpha(b'))^{-w'(\ell)}\right).
                \]
    			By definition, $\ell \in \kk(B)^\times_\calE$ implies that for all $b''\in B^0\setminus(B^0\cap\frakS)$, $\sum_{b'\in B, b'\mapsto b''} w'(\ell) \equiv 0 \bmod{2}$.  Since $\alpha(b')$ depends only on the image of $b'$ in $B^0$, this shows that the above factor is contained in $\kk(F)^{\times 2}$.

                Now consider the case that $b'$ lies over $B^0_{t_0}\cap \frakS$.  We claim that, since $w'(x) = 0$, 
                \begin{equation}\label{eq:factor}
                \Norm_{\kk(F')/\kk(F)}\left((-1)^{w'(x-\alpha) w'(\ell)}\ell^{w'(x - \alpha)}(x - \alpha)^{-w'(\ell)} \right)
                \end{equation}
                 reduces to a constant in $\kk(F)$. Indeed, if $w'(\ell) = 0$, then we obtain $\ell^{w'(x - \alpha)}$, which reduces (after taking $\Norm_{\kk(F')/\kk(F)}$) to an element of $\kk(t_0)^\times$. If $w'(\ell) \neq 0$,  let $\pi_{F'}$ be a uniformizer for $F'$. Since $w'(x) = 0 > w'(\alpha)$, we have  

    			\[
				\left(\frac{\ell}
				{\pi_{F'}^{w'(\ell)}}\right)^{w'(x - \alpha)}
				\left(\frac{x-\alpha}{\pi_{F'}^{w'(x - \alpha)}}\right)^{-w(\ell)}
				= \left(\frac{\ell}
				{\pi_{F'}^{w'(\ell)}}\right)^{w'(x - \alpha)} 
				\left(\frac{-\alpha}{\pi_{F'}^{w'(x - \alpha)}}\right)^{-w(\ell)} \bmod{\pi_{F'}} 
   			\]
    			and so~\eqref{eq:factor} reduces to (again, after taking $\Norm_{\kk(F')/\kk(F)}$) an element in $\kk(t_0)^\times$. 	{Thus, every factor of~\eqref{eq:resatFv2} corresponding to points $b'$ lying over $B^0_{t_0}\cap \frakS$ is contained in $\kk(t_0)^{\times}$, and every other factor is an element of $\kk(F)^{\times2}$. This completes the proof.}
    		\end{proof}
	    \subsection{The morphism $\beta$}\label{sec:DefinitionOfBeta}
        
            \begin{prop}\label{prop:DefinitionOfBeta}
                Let $\ell\in \kk(B)^{\times}$.  There exists an element $\calA' = \calA'(\ell)\in \Br k(t)$, unique modulo $\Br k$, such that
                \[
                    \calA_{\ell} + \varpi^*\calA' \in \Br V.
                \]
                This induces a well-defined homomorphism
                \[
                    \beta \colon \frac{\kk(B)_{\calE}^{\times}}{k^{\times}\kk(B)^{\times2}} \to \frac{\Br^{\textup{g.unr.}} U}{\Br k}[2], \quad \ell \mapsto \pi^*\left(\calA_{\ell} + \varpi^*\calA'\right),
                \]
                which is surjective if $k$ is separably closed.
            \end{prop}
            \begin{proof} 
                As a subgroup of $\Br \kk(S) = \Br \kk(S^0)$, $\Br V$ is equal to $\Br S^0 \setminus (\frakS\cup S^0_{\infty}\cup B)$, since the Brauer group of a surface is unchanged under removal of a codimension $2$ closed subscheme~\cite[Thm. 6.1]{Gro-DixExp}.  Thus, to prove the first statement, it suffices to show that there exists an element $\calA'\in \Br k(t)$, unique up to constant algebras, such that $\partial_F(\calA_{\ell}) = \partial_F(\varpi^*\calA')$ for all irreducible curves $F\subset S^0$ with $F\neq \frakS, S^0_{\infty}, B.$
                
                If $F$ is any horizontal curve, i.e., $F$ maps dominantly to $\PP^1_t$, then $\partial_F(\varpi^*\calA') = 1$ for all $\calA'\in \Br k(t)$.  If we further assume that $F\neq \frakS, B$,  then Lemma~\ref{lem: HorizResidues}  gives $\partial_F(\calA_{\ell}) = 1$.  Thus, for all $\calA'\in \Br k(t)$, we have $\partial_F(\calA_{\ell}) = \partial_F(\varpi^*\calA')$ for all horizontal curves $F\neq \frakS, B$.
                
                Now we turn our attention to the vertical curves.  Recall Faddeev's exact sequence~\cite[Cor. 6.4.6]{GS-csa}:
                 \begin{equation}\label{eq:faddeev} 
                 0 \to \Br k \to \Br k(t) \xrightarrow{\oplus \partial_{t_0}} \bigoplus_{t_0 \in \PP^1_{t}} \HH^1(G_{\kk(t_0)}, \mathbb{Q}/\mathbb{Z}) \xrightarrow{\sum_{t_0} \Cor_{\kk(t_0)/k}} \HH^1(G_k, \mathbb{Q}/\mathbb{Z})  \to 0.
                 \end{equation}
                 Since the residue field at $t_0 = \infty$ is equal to $k$, this sequence implies that for any element $(r_{t_0})\in \oplus_{t_0 \in \A^1} \kk(t_0)^{\times}/\kk(t_0)^{\times2}$, there exists a Brauer class in $\calA'\in\Br k(t)$, unique modulo elements of $\Br k$, such that $\partial_{t_0}(\calA') = r_{t_0}$.  By Lemma~\ref{lem: VertResidues}, for all $t_0\in \A^1$, we have $\partial_{F}(\calA_{\ell}) \in \im\left(\kk(t_0)^{\times}/\kk(t_0)^{\times2}\to \kk(F)^{\times}/\kk(F)^{\times2}\right)$, where $F = S^0_{t_0}$. Hence, there exists a $\calA'\in \Br k(t)$, unique modulo $\Br k$, such that $\partial_F(\varpi^*\calA') = \partial_F(\calA_{\ell})$ for all $F \neq \frakS, B, S^0_{\infty}$, as desired.
                
                It remains to prove the second statement.  The first statement immediately implies the existence of a well-defined homomorphism
                \[
                    \frac{\kk(B)_{\calE}^{\times}}{\kk(B)^{\times2}} \to \frac{\Br \pi^{-1}(V)}{\Br k}[2], \quad \ell \mapsto \pi^*\left(\calA_{\ell} + \varpi^*\calA'\right).
                \]
                In order to complete the proof, we must prove that
                \begin{enumerate}
                    \item $\pi^*\left(\calA_{d} + \varpi^*\calA'\right)\in \Br k$ if $d\in k^{\times}$,\label{cond:WellDefined}
                    \item the image lands in $\Br^{\textup{g.unr.}} U/\Br k$, and \label{cond:GeomUnr}
                    \item the image is equal to $\Br Y[2]$ if $k$ is separably closed.\label{cond:Surjective}
                \end{enumerate}
                
                We begin with {\bf\eqref{cond:WellDefined}}.  Let $d\in k^{\times}$.  Then 
                \[
                    \calA_{d} = \Cor_{\kk(B)(x)/k(t,x)}(d, x-\alpha) = 
                    (d, \Norm_{\kk(B)(x)/k(t,x)}(x - \alpha)) = (d, h(x)) = 
                    (d, c'h(x)) + (d,c').
                \]
                Since $\sqrt{c'h(x)}$ generates $\kk(Y^0)/\kk(S^0)$,  $\divv(c'h(x)) = B + 2Z$ for some divisor $Z$ on $S^0$.  Thus, $(d, c'h(x))$ is unramified away from $B$; in particular, $(d, c'h(x)) \in \Br V$.  Since $\calA'$ is the unique element in $\Br k(t)/\Br k$ such that $\calA_d + \calA'\in \Br V$, then $\calA'=(d, c') + \calB$ for some $\calB\in \Br k$.  Hence,
                \[
                    \pi^*\left(\calA_d + \varpi^*\calA'\right) = 
                    \pi^*\left( (d, c'h(x)) + (d, c') + (d, c') + \pi^*\varpi^*\calB\right)
                    = \pi^*(d, c'h(x)) + \pi^*\varpi^*\calB.
                \]
                Furthermore, since $c'h(x)$ is a square in $\kk(Y^0)$,  then  $\pi^*\left(\calA_d + \varpi^*\calA'\right) =\pi^*\varpi^*\calB\in \Br k$, as desired.
                
                Now we turn to~\eqref{cond:GeomUnr} and~\eqref{cond:Surjective}. Since $B$ is the branch locus of $\pi$ and $\pi$ is $2$-to-$1$, any $2$-torsion Brauer class in $\im \left(\pi^\ast \colon \Br \kk(S) \to \Br \kk(Y) \right)$ is unramified at $\pi^{-1}(B)_{\textup{red}}$.  Thus, the image is contained in $\Br U/\Br k$.  To prove that it is contained in $\Br^{\textup{g.unr.}} U$, we must show that $\pi^*\left(\calA_{\ell} + \varpi^*\calA'\right)_{\kbar}$ is contained in $\Br \Ybar$.  By Tsen's theorem, $\pi^*\left(\calA_{\ell} + \varpi^*\calA'\right)_{\kbar} = (\pi^*\calA_{\ell})_{\kbar}$.  This element is contained in $\Br \Ybar$ by~\cite[Thm. I]{CV-BrauerSurfaces}, which yields~\eqref{cond:GeomUnr}.  In fact,~\cite[Thm. I]{CV-BrauerSurfaces} shows that $\Br \Ybar[2]$ is generated by $\pi^*\calA_{\ell}$ where $\ell$ runs over the elements in $\kk(B_{\kbar})_{\calE}$, which proves~\eqref{cond:Surjective}.
            \end{proof}

    	\subsection{The morphism $j$ }\label{sec:DefinitionOfj}
            In this section, we define the map $j$ and prove that it is injective.  The map $j$ will be induced by the following homomorphism:
	        \begin{align*}
	            j' \colon
	            \Div(Y\setminus\pi^{-1}(B)) &\to  \kk(B)^{\times}/k^{\times}\\
	            D & \mapsto \ell|_B 
	        \end{align*}
            where $\ell\in\kk(S)^{\times}$ is such that $\divv(\ell) = \pi_*D - m_1E_1 - \dots - m_nE_n - d\frakS - eS^0_\infty.$

    	    \begin{lemma}\label{lem: InjectivityOfj}
    	     	The homomorphism $j'$ induces a well-defined injective homomorphism
    	        \[
    	        	j: \frac{\Pic Y}{\pi^* \Pic S + 2 \Pic Y} \to \frac{\kk(B)_{\calE}^{\times}}{k^{\times}\kk(B)^{\times2}}.
    	        \]
    	    \end{lemma}

    	    \begin{proof}
    	        For any divisor $D\in \Div Y\setminus \pi^{-1}(B)$, the projection formula~\cite[p.399]{QingLiu} yields
    				\[
    					2\pi_*(D\cap \pi^{-1}(B)_{\textup{red}}) = \pi_*(D\cap 2\pi^{-1}(B)_{\textup{red}}) =  \pi_*(D\cap \pi^{*}(B)) = (\pi_*D)\cap B.
    				\]
    			Thus, for any divisor $D\in \Div Y\setminus\pi^{-1}(B)$, we have that $[D\cap \pi^{-1}(B)_{\textup{red}}]\in\left(\frac{\Pic B}{\im \Pic S\to \Pic B}\right)[2]$.  By the same argument as in proof of~\cite[Lemma 4.8]{IOOV-UnramifiedBrauer}, this induces a well-defined injective homomorphism

    			\begin{equation}\label{eq:PicYtoPicB}
    				\frac{\Pic Y}{\pi^*\Pic S + 2\Pic Y} \to \left(\frac{\Pic B}{\im \Pic S\to \Pic B}\right)[2], \quad [D] \mapsto [D\cap \pi^{-1}(B)_{\textup{red}}].
    			\end{equation}

    			One can also check that there is a well-defined injective  homomorphism
    			\begin{equation}\label{eq:PicBtoFunctions}
    				\left(\frac{\Pic B}{\im \Pic S\to \Pic B}\right)[2] \to \frac{\kk(B)_{\calE}^{\times}}{k^{\times}\kk(B)^{\times2}}
    			\end{equation}
    			that sends a divisor $D$ which represents a class in $\left(\frac{\Pic B}{\im \Pic S\to \Pic B}\right)[2]$ to a function $\ell$ such that $\divv(\ell) = 2D + \sum_{C\in \Pic S}n_C C\cap B$.  Since $j$ is the composition of~\eqref{eq:PicYtoPicB} and~\eqref{eq:PicBtoFunctions}, this completes the proof that $j$ is well-defined and injective.
    	    \end{proof}

	    \subsection{Brauer classes on $V$ that become constant under $\pi^*$}
        \label{sec:KernelPullback}

    		\begin{prop}\label{prop:KernelPullback}
        	 	If  $\calA \in \Br \opendP$ is such that $\pi^*\calA \in \Br k \subset \Br\kk(Y_k)$, then there exists a divisor $D\in \Div Y_k$ such that $\xminusalpha([D])|_B = \partial_B(\calA)$ in $\kk(B)^{\times}/k^{\times}\kk(B)^{\times2}.$
        	\end{prop}

        	\begin{proof}
        		Recall that $\kk(Y_k) = \kk(\dP_k)(\sqrt{c'h(x)})$. Thus, if $\pi^*\calA\in \Br k$, then
    			\begin{equation}\label{eq:kernelBr}
    				\calA = (c'h(x), G) + \calB
    			\end{equation}
    		 	for some $G\in \kk(S_k)^{\times}$ and some $\calB \in \Br k$.  Since $B$ is the branch locus of $\pi$, $v_B(c'h(x))$ must be odd.  Therefore, without loss of generality, we may assume that $B$ is not contained in the support of $G$; write
    			\[
    				\divv(G) = \sum_i n_i C_i + d(\frakS) + e(S_\infty^0) + m_1E_1 + \dots + m_n E_n,
    			\]
    			where $C_i$ are $k$-irreducible curves of $S$ distinct from $\frakS, S^0_{\infty}$, and $E_1, \dots, E_n$.

    			Now we consider the residue of $\calA$ at $C_i$. By~\eqref{eq:kernelBr}, the residue of $\calA$ at $C_i$ is $\left[ c'h(x) \right]\in \kk(C_i)^{\times}/\kk(C_i)^{\times2}$. On the other hand, $\calA \in \Br \opendP$, so the residue is trivial at $C_i$.  Together, these statements imply that $\pi^{-1}(C_i)$ consists of $2$ irreducible components $C_i'$ and $C_i''$.  As this is true for all $C_i$, we have that $\divv(G) = \pi_*(\sum_i n_i C_i') + m(\frakS) + m_0(S_\infty^0) + m_1E_1 + \dots + m_n E_n$, and so $j'(\sum n_i C_i') = G|_B$ modulo $k^{\times}$.  Since the residue of $\calA$ at $B$ is equal to $G|_B$, this completes the proof.
            \end{proof}

	    \subsection{Proof of Theorem~\ref{thm:NumberFieldExactSequence}}
        \label{sec:ProofOfNumberFieldExactSequence}
			We note that much of this proof is very similar to proofs in~\cite[Lemmas 4.4 and 4.8]{IOOV-UnramifiedBrauer}.  
            
            We will first prove the sequence is exact, and then show that the maps are compatible with the Galois action.  Since all assumed properties of $k$ are preserved under field extension, we may, for the moment, assume that $k = k'$.  Then Lemma~\ref{lem: InjectivityOfj} yields an injective homomorphism
            \[
		        j \colon \frac{\Pic Y_{k'}}{\pi^*\Pic S + 2\Pic Y_{k'}} \To 
		            \frac{\kk(B_{k'})_{\calE}^{\times}}{k'^{\times}\kk(B_{k'})^{\times2}},
            \]
            and Proposition~\ref{prop:DefinitionOfBeta} yields a homomorphism
            \[
                \beta\colon  \frac{\kk(B_{k'})_{\calE}^{\times}}{k'^{\times}\kk(B_{k'})^{\times2}}
					\To
					\left(\frac{\Br^{\textup{g. unr.}} U_{k'}}{\Br k'}\right)[2],
            \]
            which is surjective if $k'$ is separably closed.
			We now show that $\im(j) = \ker(\beta)$. 

Let $\ell \in \kk(B_{k'})_{\calE}^{\times}$ be such that $\beta(\ell)\in \Br k'$.  Recall that $\beta$ factors through $\Br V$ by the map
            \[
                \ell \mapsto \underbrace{\calA := \calA_{\ell} + \varpi^*\calA'}_{\in \Br V} \mapsto \pi^*\calA,
            \]
            where $\calA'\in \Br k'(t)$ is as in Proposition~\ref{prop:DefinitionOfBeta}.  By assumption, $\pi^*\calA \in \Br k$, thus, by Proposition \ref{prop:KernelPullback}, there is some $D \in \Div Y_k$ such that $j([D])|_B = \partial_B(\calA) = \partial_B(\calA_\ell)  \partial_B(\varpi^*\calA')\bmod{k^{\times}}$.  However, $\partial_B(\varpi^*\calA') = 1$ since $B$ is a horizontal divisor, and $\partial_B(\calA_\ell) = [\ell]$ by Lemma~\ref{lem: HorizResidues}.  Hence, $\ell \in \im(j)$, and so $\im(j) \supset \ker(\beta)$.

			For the opposite inclusion, it suffices to prove that $\beta(j([D])) \in \Br k'$ for any prime divisor $D\in \Div(Y_{k'}\setminus \pi^{-1}(B))$. 
			Let $\ell = j'(D)$; recall that $\ell$ is the restriction to $B$ of a function $\ell_S\in \kk(S_{k'})$ such that $\divv(\ell_S) = \pi_\ast D - m_1E_1 - \dots - m_nE_n - d\frakS - eS^0_\infty$.
			As above, let $\calA := \calA_\ell + \varpi^*\calA'$.  We claim that
			\[
				\calA = (c'h(x), \ell_S) + \calB \quad \in \Br\kk(S_{k'}) = \Br \kk(S^0_{k'})
			\]
			for some $\calB\in \Br k'$. Since $c'h(x)\in\kk(Y_{k'})^{\times2}$, this equality implies that $\pi^\ast(\calA) = \pi^\ast \calB \in \Br k'$.  To prove the claim, we will compare residues of $\calA$ and $(c'h(x), \ell_S)$ on $S^0$. 
            Repeated application of Faddeev's exact sequence~\cite[Cor. 6.4.6]{GS-csa} shows that $\Br \A^n_{k'}$ is trivial; in particular, the Brauer group of $S^0 \setminus (S^0_{\infty} \cup \frakS)$, which is isomorphic to $\A^2$, consists only of constant algebras.
             Hence, if $\calA- (c'h(x), \ell_S)$ is unramified everywhere on $S^0 \setminus (\frak S \cup S^0_\infty)$, then it must be constant. By Lemma~\ref{lem: HorizResidues}(2) and the assumption that $\calA \in \Br V$, it suffices to show that $\partial_B((c'h(x), \ell_S)) = [\ell]$ and that $(c'h(x), \ell_S)$ is unramified along all other curves irreducible curves contained in $V$. 

             Let $R$ be a prime divisor of $S^0$ different from $S^0_{\infty}, \frakS$, and $B$. Since $B$ is the branch locus of $\pi$, we may assume that $v_R(c'h(x)) =0.$
       			  Hence, $\partial_R((c'h(x), \ell_S)) = (c'h(x))^{v_R(\ell_S)}$. 
			  Now, we know that 
              \[
              \divv_S(\ell_S) = \pi_\ast D - m_1E_1 - \dots - m_nE_n - d\frakS - eS^0_\infty.
              \]
              Thus, if $R \ne \pi(D)$ then $v_R(\ell_S) = 0$ and hence $\partial_R((c'h(x), \ell_S))$ is trivial.  It remains to consider the case $R = \pi(D)$. Note that $\pi_\ast(D)$ is equal to  $\pi(D)$ if $\pi$ maps $D$ isomorphically to its image, and is equal to $2\pi(D)$ otherwise. In the latter case, we must have that $v_{\pi(D)}(\ell_S)$ is even, meaning that $\partial_{\pi(D)}((c'h(x), \ell_S))$ is trivial (up to squares). On the other hand, if  $\pi_\ast(D)= \pi(D)$, then  $c'h(x)$ must be a square in $\kk(\pi(D))$, and hence $\partial_{\pi(D)}((c'h(x), \ell_S))$ is again trivial (up to squares).

			Finally, $\partial_B((c'h(x), \ell_S)) = [\ell_S|_B] = [\ell]$. Indeed, since we know that 
			\[
			\divv(\ell_S) = \pi_\ast D - m_1E_1 - \dots - m_nE_n - d\frakS - eS^0_\infty
			\]
			 for some $D\in \Div(Y\setminus \pi^{-1}(B))$, we see that $B$ is not in the support of $\divv(\ell_S)$.  Hence $v_B(\ell_S) =0$.  Moreover, since $v_B(c'h(x))$ is odd,  the usual residue formula  allows us to deduce that   $\partial_B((c'h(x),\ell_S)) =[ \ell]$ modulo squares.  Hence, $\calA = (c'h(x), \ell_S) + \calB$ and the sequence is exact, as desired.
            
            Now we consider the Galois action.  That $j$ respects the Galois action is clear from the definition.  To see that $\beta$ is a homomorphism of Galois modules, we note that every geometrically irreducible curve outside of $V$ is irreducible over $k$.  Thus, the residue maps $\partial_F$ for $F$ outside of $V$ are defined over $k$, which shows that $\beta$ is a homomorphism of Galois modules.  This completes the proof.
            \qed

    \subsection{Proof of Theorem~\ref{thm:Br1vsBr}}
    \label{sec:ProofOfThmBr1vsBr}
        Applying Theorem~\ref{thm:NumberFieldExactSequence} to $\kbar$ and taking the subgroups of Galois invariant elements gives an exact sequence
        \begin{equation}\label{eq:ExactSequencekbar}
            0 \To \left(\frac{\Pic \Ybar}{\pi^*\Pic S + 2\Pic \Ybar}\right)^{G_{k'}} 
            \stackrel{j_{\kbar}}{\To}
            \left(\frac{\kk(\overline{B})_{\calE}}
            {\kk(\overline{B})^{\times2}}\right)^{G_{k'}}
            \stackrel{\beta_{\kbar}}{\To}
            (\Br \Ybar)^{G_{k'}}.
        \end{equation}
        Recall that $j_{\kbar}$ factors through $\kk(\Sbar)/\kk(\Sbar)^{\times2}$, so the middle term may be replaced with $\left({\kk({\overline{B}})_{\calE}}/
            {\kk(\overline{B})^{\times2}}\right)^{G_{k'}} \cap \im(\kk(\Sbar)/\kk(\Sbar)^{\times2})^{G_{k'}}$.

        To determine $(\kk(\Sbar)/\kk(\Sbar)^{\times2})^{G_{k'}}$, we consider the exact sequence
        \[
            0 \To \kbar^{\times} \To \kk(\Sbar)^{\times} \To \kk(\Sbar)^{\times}/\kbar^{\times} \To 0.
        \]
        After taking the cohomological long exact sequence, applying Hilbert's Theorem 90, and applying the assumption that $\Br k' \To \Br \kk(S_{k'})$ is injective, we obtain
        \[
            \HH^0\left(G_{k'}, \kk({\Sbar})^{\times}/\kbar^{\times}\right) =
            \kk(S_{k'})^{\times}/k'^{\times},
            \quad\textup{ and }\quad
            \HH^1\left(G_{k'}, \kk({\Sbar})^{\times}/\kbar^{\times}\right) = 0.
        \]
        Then the cohomological long exact sequence associated to
        \[
            0 \To \kk({\Sbar})^{\times}/\kbar^{\times} \stackrel{\times2}{\To}
            \kk({\Sbar})^{\times}/\kbar^{\times} \To
            \kk({\Sbar})^{\times}/\kk({\Sbar})^{\times2}\To 0
        \]
        yields $\left(\kk({\Sbar})^{\times}/\kk({\Sbar})^{\times2}\right)^{G_{k'}} = \kk(S_{k'})/(k'^\times\kk(S_{k'})^{\times2})$.
              Thus, we may replace~\eqref{eq:ExactSequencekbar} with
            \[
            0 \To \left(\frac{\Pic \Ybar}{\pi^*\Pic S + 2\Pic \Ybar}\right)^{G_{k'}} 
            \stackrel{j_{\kbar}}{\To}
            \frac{\kk(B_{k'})_{\calE}}
            {k'^{\times}\kk(B_{k'})^{\times2}}
            \stackrel{\beta_{\kbar}}{\To}
            (\Br \Ybar)^{G_{k'}}.            
            \]
            Note that $\beta_{\kbar}|_{\kk(B_{k'})_\calE}$ factors through $\Br^{\textup{g.unr.}} U_{k'}/\Br k'$.  Hence, we obtain the following commutative diagram:
            \[
                \xymatrix{            
                    & \frac{\Pic Y_{k'}}{\pi^*\Pic S + 2\Pic Y_{k'}} \ar[r]^(.6)j \ar@{^{(}->}[d]&
                    j(\Pic Y_{k'})\ar@{^{(}->}[d]\\
                    0\ar[r]&
                    \left(\frac{\Pic \Ybar}{\pi^*\Pic S + 2\Pic \Ybar}\right)^{G_{k'}}
                    \ar[r]^(.6)j \ar^{\beta\circ j}[d] & 
                    \frac{\kk(B_{k'})_{\calE}}{k'^{\times}\kk(B_{k'})^{\times2}}
                    \ar^{\beta}[r]\ar^\beta[d]&
                   \frac{\Br^{\textup{g.unr.}} U_{k'}}{\Br_1 U_{k'}} \ar@{=}[d]\\
                    0 \ar[r] & \frac{\Br_1 U_{k'}}{\Br k'} \ar[r] & 
                    \frac{\Br^{\textup{g.unr.}} U_{k'}}{\Br k'} \ar[r] &
                    \frac{\Br^{\textup{g.unr.}} U_{k'}}{\Br_1 U_{k'}} \ar[r] & 0.\\
                }
            \]
            It remains to prove that the rows and columns are exact and that the leftmost bottom vertical arrow is surjective.  The exactness of the middle row follows from the above discussion and the exactness of the bottom row follows from the definitions. The middle column is exact by Theorem~\ref{thm:NumberFieldExactSequence}; Theorem~\ref{thm:NumberFieldExactSequence} and Proposition~\ref{prop:KernelPullback} together imply that the leftmost row is exact.  
            
            Consider the map induced by $\beta\circ j$
            \[
                \frac{\left({\Pic \Ybar}/({\pi^*\Pic S + 2\Pic \Ybar})\right)^{G_{k'}}}{\Pic Y_{k'}/(\pi^*\Pic S + 2\Pic Y_{k'})} \hookrightarrow \frac{\Br_1 U_{k'}}{\Br k'};
            \]
            we would like to show that it is surjective.
            Note that $\Pic \Ubar \isom {\Pic \Ybar}/{\pi^*\Pic S}$ and $\Pic U_{k'} \isom \Pic Y_{k'}/\pi^*\Pic S$.  Furthermore, since $\Pic \Ubar$ is torsion free, the Hochschild-Serre spectral sequence together with the cohomological long exact sequence associated to the multiplication by $2$ map yields the isomorphism
            \[
                \frac{(\Pic \Ubar/2\Pic \Ubar)^{G_{k'}}}
                {(\Pic \Ubar)^{G_{k'}}/(2\Pic \Ubar)^{G_{k'}}}\cong\frac{\Br_1 U_{k'}}{\Br k'}.
            \]
            Since all groups in question are finite, a cardinality argument completes the proof of surjectivity, and hence the proof of the theorem.\qed

%% file: GeometryDoubleCoverings.tex
\ifx
    \thepage\undefined\def\jobname{BBMPV_EnriquesTranscendental}
    \input{BBMPV_EnriquesTranscendental}
\fi

    \section{Geometry of $\Xabc$ and $\Yabc$}\label{sec:Geometry}
        
        \subsection{Review of~\cite[\S4]{VAV-Enriques}}
            The K3 surface $\Yabc$ is defined as the base locus of a net of quadrics.  As explained in~\cite[Example IX.4.5]{Beauville-Surfaces} and~\cite[\S4.1]{VAV-Enriques}, each isolated singular point in the degeneracy locus of the net gives rise to two genus $1$ fibrations on $\Yabcbar$.  As the degeneracy locus of the net has $14$ singular points, we obtain $28$ classes of curves in $\Pic \overline{\Yabc}$, which we denote $F_1, G_1, \dots, F_{14}, G_{14}$; for all $i, j$ we have the relation $F_i + G_i = F_j + G_j$.

            Nine of these singular points define fibrations which descend to $\Xabcbar$.  On $\Xabcbar$, these fibrations have exactly two nonreduced fibers.  For a fixed point $P_i$ in the degeneracy locus, we let $C_i, \Ctilde_i, D_i,$ and $\Dtilde_i$ denote the reduced subschemes of the nonreduced fibers of the corresponding fibrations.  After possibly relabeling, we may assume that we have the following relations in $\Pic \Xabcbar$:
            \[
                C_i + D_i = \Ctilde_j + \Dtilde_j, \quad 2(C_i - \Ctilde_i) = 2(D_i - \Dtilde_i) = 0,  \quad f^*C_i = f^*\Ctilde_i = F_i, \quad f^*D_i = f^*\Dtilde_i = G_i.
            \]
            Defining equations for curves representing each of these classes is given in Table~\ref{table:Picard}.  The intersection numbers of these curves are as follows:
            \[
            \begin{array}{lllll}
                F_i^2 = G_i^2 = 0,& \quad&
                F_i\cdot G_i = 4, & \quad & 
                F_i\cdot G_j = F_i \cdot F_j = G_i \cdot G_j = 2
                \;\textup{for all }i\neq j,\\
                C_i^2 = D_i^2 = 0,& \quad&
                C_i\cdot D_i = 2, & \quad & 
                C_i\cdot D_j = C_i \cdot C_j = D_i \cdot D_j = 1
                \;\textup{for all }i\neq j.\\
            \end{array}
            \]
            
        \begin{prop}[{\cite[Cor 4.3]{VAV-Enriques}}]
            Let $\mathbf{a}\in \Z^3_{>0}$ satisfy conditions $(5)$, $(6)$, and $(8)$.  Then $\Pic \Yabcbar \isom \Z^{15}$ and is generated by $G_1, F_1, \dots, F_{14}$ and 
            \begin{align*}
                Z_1 := \frac12\left(F_1 + F_2 + F_3 + F_{10} + F_{12}\right), 
                & \quad
                Z_2 := \frac12
                \left(F_1 + G_1 + F_4 + F_5 + F_6 + F_{10} + F_{11}\right),\\
                Z_3 := \frac12\left(F_1 + F_4 + F_7 + F_{13} + F_{14}\right), 
                & \quad
                Z_4 := \frac12
                \left(F_1 + G_1 + F_7 + F_8 + F_9 + F_{11} + F_{12}\right).
            \end{align*}
        \end{prop}

        As in~\cite{VAV-Enriques}, we let $K/\Q$ denote the splitting field of the curves $F_i, G_i$.  We will be concerned with two particular subfields $K_0\subset K_1 \subset K$; we give generators for these fields in Appendix~\ref{app:Fields} and describe their defining properties in~\S\S\ref{sec:DoubleCovers}--\ref{sec:BranchLoci}.
        \subsection{Double cover morphisms}\label{sec:DoubleCovers}
            In order to apply the results of~\S\ref{sec:Pullback}, we must realize $\Yabc$ as a double cover of a rational ruled surface.  We will be able to do so over the Galois extension $K_0:= \Q(i, \sqrt2,\sqrt5, \sqrt{-2 + 2\sqrt2})$; throughout this section, we work over $K_0$.
            
            Consider the morphism 
        \[
            \phi\colon \Yabc \to \quadric :=\PP^1_x \times \PP^1_t,
        \]
	    which sends a point $[v_0:v_1:v_2:w_0:w_1:w_2]$ to
        \begin{equation}\label{eq:phi}
           \left(\left[w_0 - \sqrt5 w_1 : v_0 + 2v_1 \right], \;
           \left[\sqrt{-2 + 2\sqrt2}w_0 - \sqrt5w_1,
		   v_0 + \sqrt2v_1 + \sqrt5(1 - \sqrt2)v_2 \right]\right).
        \end{equation}
        For any $P\in \Yabc$, we have $\phi(\sigma(P)) = [-1](\phi(P))$, where $[-1]$ denotes the automorphism of $\quadric$ that sends $(x,t)$ to $(-x,-t)$.  Thus, we have an induced morphism
        \[
            \tilde \phi\colon X_{\mathbf{a}} \to \left(\quadric\right)/[-1],
        \]
        obtained by quotienting $\Yabc$ by $\sigma$ and $S^0$ by $[-1]$.  
        
        The morphism $\phi$ factors through a double cover morphism $\pi \colon \Yabc \to \dP,$ where $\dP := \Yabc/(w_2 \mapsto - w_2)$.  The quotient $\dP$ is a smooth del Pezzo surface of degree $4$ given by
        \[
            \left\{v_0v_1 + 5v_2^2 - w_0^2 = v_0^2 + 3v_0v_1 + 2v_1^2 - w_0^2 + 5w_1^2 = 0\right\}\subset\PP^4_{(v_0:v_1:v_2:w_0:w_1)}.
        \]
        Using~\eqref{eq:phi}, one can check that $\phi$ induces a birational map $\BldP\colon \dP\to \quadric$ which contracts four $(-1)$-curves; we denote these curves by $E_1,E_2,E_3,$ and $E_4$. (Defining equations for the curves are given in Table~\ref{table:ExceptionalCurves}.)  
        
        The preimages of the $E_i$ under $\pi$ are irreducible $(-2)$-curves in $\Yabc$.  Thus, we may also factor $\phi$ by first blowing-down these four $(-2)$-curves to obtain a (singular) surface $\Yabc^0$ and then quotienting by an involution to obtain a double cover morphism.  Hence, we have the following commutative diagram,
        \[
            \xymatrix{
                \Yabc \ar^{\BlY}[d]\ar^{\pi}[r]\ar^{\phi}[rd]& \dP\ar^{\BldP}[d]\\
                \Yabc^0\ar^{\pi^0}[r] &\quadric
            }
        \]
        where the vertical maps are birational and the horizontal maps are $2$-to-$1$. Note that over $k = K_0$, these varieties and morphisms satisfy all assumptions from~\S\ref{sec:Pullback}.  
        In particular, all morphisms are defined over $K_0$ and $\Pic S_{K_0} = \Pic \Sbar.$
        
        Since $\sigma$ induces an involution on $\Yabc^0$, $\dP$, and $\quadric$ which is compatible with the morphisms, the above commutative diagram descends to the following commutative diagram involving the Enriques surface:
        \[
            \xymatrix{
                \Xabc \ar^{\BlX}[d]\ar^{\tilde\pi}[rr]\ar^{\tilde\phi}[rrd]&& \tilde\dP\ar^{\tildeBldP}[d]\\
                \Xabc^0\ar^{\tilde\pi^0}[rr] &&(\quadric)/{[-1]}
            }
        \]
                
        \subsection{Branch loci of the double covers}\label{sec:BranchLoci}
        
            Let $B, B^0, \Btilde, \Btilde^0$ denote the branch loci of $\pi, \pi^0, \tilde\pi$ and $\tilde\pi^0$ respectively.  
            \begin{prop}
            \label{prop:BranchLoci}
			\hfill
            \begin{enumerate}
                \item $B = V(w_2) \subset \PP^4$ is a smooth genus $5$ curve,
                \item $\Btilde$ is a smooth genus $3$ curve,
                \item $B^0$ is an arithmetic genus $9$ curve with $4$ nodal singularities, and
                \item $\Btilde^0$ is an arithmetic genus $5$ curve with $2$ nodal singularities.
                \end{enumerate}
                Furthermore, the projection morphism $\Btilde \to V(av_0^2 + bv_1^2 + cv_2^2)\subset \PP^2$ is a double cover map, so $\Btilde$ is geometrically hyperelliptic.
            \end{prop}
            
            \begin{proof}
                From the definition of $\pi$, one can immediately see that $B$ is the image of $V(w_2)\cap\Yabc$, so $B$ is given by an intersection of three quadrics in $\PP^4.$  Since~\eqref{eq:Nonsingular} holds, $B$ is smooth by the Jacobian criterion.  Therefore,  $B$ is a smooth genus $5$ curve.     
            
            As $B/\sigma = \Btilde$ and $\sigma|_{B}$ has no fixed points, the Riemann-Hurwitz formula implies that $\Btilde$ is a smooth genus $3$ curve.  Furthermore, the curve $B$ has a $4$-to-$1$ projection map to the plane conic $\left\{av_0^2 + bv_1^2 + cv_2^2 = 0\right\}\subset \PP^2$.  This induces a double cover map from $\Btilde$ to the same conic, so $\Btilde$ is (geometrically) hyperelliptic.
            
             Using the equations given in Table~\ref{table:ExceptionalCurves}, one can check that each $(-1)$-curve $E_i$ intersects $B$ transversely in $2$ distinct points.  Thus, the curve $B^0$ has four nodal singularities, and so has arithmetic genus $9$.  As $B^0 \to \Btilde^0$ is an \'etale double cover, $B^0$ has two nodal singularities and the Riemann-Hurwitz formula implies that $\Btilde^0$ has arithmetic genus $5$.
            \end{proof}   
            
            Since $\Btilde$ is geometrically hyperelliptic, we may identify $\Jac(\Btilde_{\Qbar})[2]$ with subsets of the Weierstrass points of $\Btilde$ of even order, modulo {identifying complementary subsets}.  Under this identification, the group operation is given by the symmetric difference, i.e., $A + B = (A\cup B)\setminus (A \cap B)$.  One can easily see that the ramification locus of the projection morphism $B\to V(av_0^2 + bv_1^2 + cv_2^2)\subset \PP^2$ is given by $V(w_0)\cup V(w_1)\subset B$.  Since this projection morphism factors as
            \[
                B\To \Btilde \To V(av_0^2 + bv_1^2 + cv_2^2)\subset \PP^2,
            \]
            where the first morphism is \'etale, the Weierstrass points of $\Btilde$ are $f(V(w_0))\cup f(V(w_1))$.  There are four points in $f(V(w_0))\subset \Btilde$, which we denote by $P_1,P_2,P_3$ and $P_4$, and four points in $f(V(w_1))\subset\Btilde$, which we denote by $Q_1,Q_2,Q_3,$ and $Q_4.$  We let
            \[
                K_1/K_0:= \textup{the splitting field of the Weierstrass points.}
            \]
            
            \begin{prop}\label{prop:KernelJac}\hfill
                \begin{enumerate}
                    \item The kernel of $f^*\colon \Jac(\Btilde_{\Qbar})[2] \to \Jac(B_{\Qbar})[2]$ is isomorphic to $\Z/2\Z$, and the unique nontrivial element is $\{P_1,P_2,P_3,P_4\}$.  
                    \item The image of $f^*$ fits in the following \emph{non-split} exact sequence of $G_{K_0}$-modules:
                    \[
                        0 \to  \textup{Ind}^{K_0}_{K_0(\theta_0)}(\Z/2\Z)\times
                    \textup{Ind}^{K_0}_{K_0(\sqrt{ab})}(\Z/2\Z) \to \im f^*
                    \to \Z/2\Z\to 0.
                    \]
                \end{enumerate}
            \end{prop}
            \begin{proof}
		        {\bf{(1)}} Let $P_i'$ and $P_i''$ be the inverse image of $P_i$ under the map $B \to \Btilde$. The vanishing locus $V(w_0)$ on $B$ consists of $P_i', P_i''$ for $i = 1, \dots, 4$.  For any distinct pair of numbers $i_0,i_1\in\{1,2,3,4\}$, there exists a linear form $L = L(v_0,v_1,v_2)$ such that $V(L) \subset \PP^2$ contains the images of $P_{i_0}$ and $P_{i_1}$.  Since the points $P_i$ are Weierstrass points of $\Btilde$ and the map $B\to \Btilde$ is \'etale, this implies that
                \[
                    \divv_{B}(w_0/L) = \sum_{i=1}^4 (P_i + P_i') - 2(P_{i_0}' + P_{i_0}'' + P_{i_1}' + P_{i_1}'').
                \]
                Thus $f^*\{P_1,P_2,P_3,P_4\}$ is principal on $B$.
                
                Now let $D$ be a divisor such that $[D] \in \Jac \Btilde[2]$ is a nontrivial element of $\ker f^*$.  Then $f^{-1}(D) = \divv_{B}(g_0)$, for some function $g_0\in\kk(B)$.  On the other hand, $2D = \divv_{\Btilde}(g_1)$ for some function $g_1\in \kk(\Btilde)$.  Thus, $g_1 = g_0^2$ in $\kk(B)$.  Since $\kk(B)$ is a quadratic extension of $\kk(\Btilde)$ generated by $w_0/L$ for $L = L(v_0,v_1,v_2)$ a linear form, this implies that $g_0 = h\cdot(w_0/L)$ in $\kk(B)$ for some $h\in \kk(\Btilde)^{\times2}$. (Note that $w_0/L\notin \kk(\Btilde)$.) Hence $D = \{P_1,P_2,P_3,P_4\}$ in $\Jac\Btilde.$  
                
                {\bf{(2)}} By $(1)$ and the description of $\Jac \Btilde[2]$ in terms of Weierstrass points, we see that the image of $f^*$ is isomorphic to $(\Z/2\Z)^5$ (as a group) and is generated by
                \[
                    \{P_1, P_3\}, \; \{P_1, P_4\}, \; 
                    \{Q_1, Q_3\}, \; \{Q_1, Q_4\}, \; \textup{ and }\;
                    \{P_1,Q_1\}.
                \]
                Using Table~\ref{table:Galois}, one can check that (as $G_{K_0}$-modules) we have:
                \[
                    \langle \{P_1, P_3\}, \; \{P_1, P_4\}\rangle \isom 
                    \textup{Ind}^{K_0}_{K_0(\sqrt{ab})}(\Z/2\Z), 
                    \quad \textup{and}\quad 
                    \langle \{Q_1, Q_3\}, \; \{Q_1, Q_4\}\rangle \isom 
                    \textup{Ind}^{K_0}_{K_0(\theta_0)}(\Z/2\Z),
                \]
                thus proving the existence of the exact sequence.  Since $\Gal(K_1/K_0)$ acts transitively on $\{P_1,P_2,P_3,P_4\}$, no element of the form $\{P_i, Q_j\}$ is fixed by $G_{K_0}$ and so the exact sequence does not split.
            \end{proof}
            
        \subsection{The quotient group $\Pic \Yabcbar/\left(\pi^*\Pic S %
         + 2\Pic\Yabcbar\right)$}\label{sec:Picard}
            The following lemma about the structure of $\Pic \Yabcbar/\left(\pi^*\Pic S %
         + 2\Pic\Yabcbar\right)$ as a Galois module will be useful in later sections.
            \begin{lemma}\label{lem:GaloisPicYbar}
                We have an isomorphism of $\Gal(K_1/K_0)$-modules
                \[
                    \left(\frac{\Pic \Yabcbar}
                    {\pi^*\Pic S + 2\Pic\Yabcbar}\right)^{G_{K_1}} \stackrel{\sim}{\To}\Z/2\Z\times \textup{Ind}^{K_0}_{K_0(\theta_0)}(\Z/2\Z)\times
                    \textup{Ind}^{K_0}_{K_0(\sqrt{ab})}(\Z/2\Z).
                \]
            \end{lemma}
            \begin{proof}
                By computing intersection numbers, we find that $\pi^*\Pic S$ is generated by
                \[
                    G_1,\; F_1 - F_2,\; Z_1 - F_2 - F_{10} - F_{12},\;
                    Z_1 - F_1 - F_2,\; -Z_1 + F_1 + F_{12},\;\textup{and }
                    -Z_1 + F_1 + F_{10}.
                \]
                Therefore, $\Pic \Yabcbar/\left(\pi^*\Pic S + 2\Pic\Yabcbar\right)$ is a $9$-dimensional $\F_2$ vector space with basis
                 \[
                     F_5,\; F_6,\; F_8,\; F_9,\; F_{11},\; F_{13},\; Z_2,\; Z_3,\;
                      \textup{and }Z_4.
                 \]
                 In this quotient we have the relations $G_i = F_i$, $F_1  = F_2 = F_{10} = F_{12} = Z_1 = G_1 = F_3 = 0$,
                 \[
                     F_4  = F_5 + F_6 + F_{11},\quad
                     F_7  = F_8 + F_9 + F_{11},\quad \textup{and }\quad
                     F_{14}  = F_5 + F_6 + F_8 + F_9 + F_{13}.
                 \]
                 Using this basis and Table~\ref{table:Galois}, we compute that $(\Pic \Yabcbar/\left(\pi^*\Pic S + 2\Pic\Yabcbar\right))^{G_{K_1}}$ is a $5$-dimensional $\F_2$-vector space with basis $F_5, F_6, F_8, F_9,$ and $F_{11}$.  The isomorphism then follows after noting that in the quotient $F_{11}$ is fixed by $\Gal(K_1/K_0)$, $F_5$ and $F_6$ are interchanged by $\Gal(K_0(\sqrt{ab})/K_0)$ and fixed by all other elements, and similarly for $F_8$ and $F_9$ and  $\Gal(K_0(\theta_0)/K_0)$.
            \end{proof}

%% file: TranscendentalRepresentative.tex
\ifx
    \thepage\undefined\def\jobname{BBMPV_EnriquesTranscendental}
    \input{BBMPV_EnriquesTranscendental}
\fi

    \section{A representative of the nontrivial Brauer class on $\Xabcbar$}
    \label{sec:Transcendental}
        As mentioned in the introduction, $\Br \Xabcbar \cong\Z/2\Z$.  Therefore, there is at most one nontrivial  class in $\Br \Xabc/ \Br_1 \Xabc$. To determine the existence of such a class, we must first obtain an explicit representative for the unique Brauer class in $\Br \Xabcbar$. Using Beauville's criterion~\cite[Cor. 5.7]{Beauv}, V\'arilly-Alvarado and the last author showed that if $\mathbf{a}$ satisfies conditions $(5), (6)$ and $(8)$, then
        \[
            \ker\left(f^*\colon \Br \Xabcbar \to \Br \Yabcbar\right) = 0.\quad \textup{\cite[Proof of Prop. 5.2]{VAV-Enriques}}
        \]

        Let $D$ be a divisor on $\Btilde_{\Qbar}$ such that $[D]\in \Jac(\Btilde)[2]$ and such that $[D]$ corresponds to a subset of the Weierstrass points containing an \emph{odd} number of the points $\{P_i\}_{i = 1,\dots,4}$.  Let $\tilde\ell\in \kk(\Btilde_{\Qbar})$ be such that $\divv(\tilde\ell) = 2D$.  

        \begin{prop}\label{BrauerRep}\label{prop:TransRep}
            Assume that $\mathbf{a}$ satisfies conditions $(5), (6)$, and $(8)$.  Then the Brauer class 
            \[
                \pi^*\calA_{\tilde\ell} = 
                \pi^*\Cor_{\kk(\overline{B})(x)/\Qbar(t,x)}\left((\tilde\ell, x - \alpha)_{\Qbar}\right) 
            \]
            defines an element of $\Br \Yabcbar$ and generates the order $2$ subgroup $f^*\Br \Xabcbar$.
        \end{prop}
        \begin{proof}
By~\cite[Thm 7.2]{CV-BrauerSurfaces}, the subgroup $f^*\Br \Xabc[2]\subset \Br\Yabcbar$ is generated by $\calA_{\tilde\ell'}$, where $\tilde\ell'\in \kk(\Btilde_{\Qbar})$ is such that $\nu_*(\divv(\tilde\ell')) \in 2\Div(\Btilde^0).$ Furthermore, by Theorem~\ref{thm:NumberFieldExactSequence} applied to $k' = \kbar$ we have the following exact sequence
        \[
        \begin{tikzcd}
        0 \to \frac{\Pic \Yabcbar}{\pi^*\Pic S + 2 \Pic \Ybar} \ar{r}{j} & \frac{\kk(B_{\Qbar})_{\calE}^\times}{\kk(B_{\Qbar})^{\times2}} \ar{r} & \Br \Yabcbar[2] \ar{r} & 0.
        \end{tikzcd} 
        \]
         Thus, to prove the proposition, it remains to prove that $[\tilde \ell]$ is not in the image of $j$.
         
         For convenience, we set:
         \begin{align*}
            \tilde{G} &:= \{\ell\in \kk(\Btilde_{\Qbar})^{\times} : \nu_*\divv(\ell) \in 2\Div(\Btilde^0_{\Qbar}) \}.
         \end{align*}
         We have the following commutative diagram with exact rows~\cite[Proposition 4.5]{CV-BrauerSurfaces}, where $\textup{Sing}(B_{\Qbar})$ and $\textup{Sing}(\tilde{B}_{\Qbar})$ denote the set of singular points of $B_{\Qbar}$ and $\Btilde_{\Qbar}$, respectively, and the last vertical map is the diagonal embedding on the first factor, i.e., $(m,n)\mapsto (m,m,n,n,0,0)$.
         \[
            \begin{tikzcd}
                0 \ar{r}&\Jac \Btilde_{\Qbar}[2] \ar{r}\ar{d}{f^*} &
                \frac{\tilde{G}}{\kk(\Btilde_{\Qbar})^{\times2}} 
                \ar{r}\ar{d}{f^*} & 
                (\Z/2\Z)^{\textup{Sing}(\Btilde_{\Qbar})} 
                \ar{r}\ar{d}& 0\\
                0 \ar{r}&\Jac B_{\Qbar}[2] \ar{r} &
                \frac{\kk(B_{\Qbar})_{\calE}}{\kk(B_{\Qbar})^{\times2}} \ar{r} &
                (\Z/2\Z)^{\textup{Sing}(B_{\Qbar})}\times (\Z/2\Z)^2\ar{r} & 0.
            \end{tikzcd} 
         \]
         In the rest of this section, we will view elements of $\Jac \Btilde_{\Qbar}[2]$ and $\Jac B_{\Qbar}[2]$ as elements of $\tilde{G}/\kk(\Btilde_{\Qbar})^{\times2}$ and ${\kk(B_{\Qbar})_{\calE}}/{\kk(B_{\Qbar})^{\times2}}$, respectively.

         \begin{lemma}\label{lem:Index2}
             Assume that $\mathbf{a}$ satisfies conditions $(5)$, $(6)$, and $(8)$.  Then the group $\im j\cap f^*(\Jac \Btilde_{\Qbar}[2]) $ is an index $2$ $\Gal(\Qbar/K_0)$-invariant subgroup of $f^*(\Jac \Btilde_{\Qbar}[2])$.
        \end{lemma}
        \begin{proof}
            Since $\Br \Xabcbar \isom \Z/2\Z$, the subgroup $\im j\cap f^*(\tilde G)$ has index $2$ in $f^*(\tilde G)$.  Therefore, the subgroup $\im j\cap f^*(\Jac \Btilde_{\Qbar}[2])$ has index at most $2$ in $f^*(\Jac \Btilde_{\Qbar}[2])$.  
            
            Recall that $K_1$ is splitting field of the Weierstrass points over $K_0$.  Thus $\Gal(\Qbar/K_1)$ fixes $f^*(\Jac \Btilde_{\Qbar})$ and hence $\im j\cap f^*(\Jac \Btilde_{\Qbar}[2])= (\im j)^{G_{K_1}}\cap f^*(\Jac \Btilde_{\Qbar}[2])$.  Since $j$ and $f^*$ are $G_{K_0}$-equivariant homomorphisms, the intersection is $G_{K_0}$-invariant and the elements in the intersection must have compatible $G_{K_0}$ action.  Then Proposition~\ref{prop:KernelJac} and Lemma~\ref{lem:GaloisPicYbar} together imply that the intersection is a submodule of $\textup{Ind}^{K_0}_{K_0(\theta_0)}(\Z/2\Z)\times \textup{Ind}^{K_0}_{K_0(\sqrt{ab})}(\Z/2\Z)$ and hence a proper subgroup of $f^*(\Jac \Btilde[2])$ with index equal to $2$.  
        \end{proof}

        Now we resume the proof of Proposition~\ref{prop:TransRep}.  Assume that $\tilde\ell$ is contained in the image of $j$ and hence in $\im j\cap f^*(\Jac \Btilde[2])$.  From Tables~\ref{table:WeierstrassPoints} and~\ref{table:Galois}, we see that $\Gal(\Qbar/K_0)$ acts transitively on the Weierstrass points of $\Btilde$.  Therefore, the subgroup of $\kk(B_{\Qbar})_{\calE}$ generated by $\tilde\ell$ and all of its $\Gal(\Qbar/K_0)$ conjugates contains all of $f^*(\Jac \Btilde_{\Qbar}[2])$.  Since $\im j\cap f^*(\Jac \Btilde[2])$ is $\Gal(\Qbar/K_0)$-invariant, this implies that 
        \[
        \im j\cap f^*(\Jac \Btilde_{\Qbar}[2]) = f^*(\Jac \Btilde_{\Qbar}[2])
        \]
         which contradicts Lemma~\ref{lem:Index2}.
        \end{proof}

%% file: ProofOfMainTheorem.tex
\ifx
    \thepage\undefined\def\jobname{BBMPV_EnriquesTranscendental}
    \input{BBMPV_EnriquesTranscendental}
\fi

\section{Proof of Theorem~\ref{thm:MainPrecise}}\label{sec:Proof}

	Assume that $\mathbf{a}\in\Z_{>0}^3$ satisfies conditions $(5),(6),$ and $(8)$.  The last statement of the theorem follows immediately from the first statement together with~\cite[Thm. 1.2]{VAV-Enriques}.  Thus our goal is to prove that $\Br \Xabc = \Br_1 \Xabc$.   Since $\Br \Xabc \to \Br \Xabcbar$ factors through $(\Br \Xabc_{, K_1})^{\Gal(K_1/K_0)}$, it suffices to prove that $(\Br \Xabc_{, K_1})^{\Gal(K_1/K_0)} = (\Br_1 \Xabc_{, K_1})^{\Gal(K_1/K_0)}$.  
    
    Recall from~\S\ref{sec:Transcendental} that $f^*\Br \Xabcbar$ is a non-trivial subgroup of $\Br \Yabcbar$.  So if $(\Br \Xabc_{, K_1})^{\Gal(K_1/K_0)}$ is strictly larger than $(\Br_1 \Xabc_{, K_1})^{\Gal(K_1/K_0)}$, then there exists an element $\calB\in (\Br \Yabc_{,K_1})^{\Gal(K_1/K_0)}$ such that $\calB_{\Qbar}$ is the unique non-trivial element in $f^*\Br\Xabcbar$.
    Let $\tilde\ell\in \kk(\Btilde_{\Qbar})$ be such that $\divv(\tilde\ell) = 2D$ where $[D]\in \Jac(\Btilde)[2]$ corresponds to a subset of the Weierstrass points containing an odd number of the points $P_i$.  Note that we may choose $\tilde\ell$ so that it is contained in $\kk(\Btilde_{K_1})^\times$. In what follows, we will view $\tilde\ell$ as a function on $B$ under the natural inclusion $\kk(\Btilde) \subset \kk(B).$  We remark that, {under this inclusion}, $\tilde\ell \in \kk(B_{K_1})^\times_\calE$.
	
	By  Proposition~\ref{prop:TransRep}, $f^*\Br\Xabcbar$ is generated by $(\pi^*\calA_{\tilde\ell})_{\Qbar}$, so $\calB_{\Qbar} = (\pi^*\calA_{\tilde\ell})_{\Qbar}$.  
Let $\calA':=\calA'(\tilde\ell)$ be as in Proposition \ref{prop:DefinitionOfBeta}. Then an  application of Tsen's theorem shows  $(\pi^*(\varpi^*\calA'))_{\Qbar}=0$. Hence, $\beta(\tilde\ell)_{\Qbar} = (\pi^*\calA_{\tilde\ell})_{\Qbar} = \calB_{\Qbar}$.
Also,  since $\beta(\tilde\ell)\in \Br U_{K_1}$, we must have that
	\[
		\calB - \beta(\tilde\ell)\in \Br_1 U_{K_1}.
	\]

    By Theorem~\ref{thm:Br1vsBr}, $\Br_1 U_{K_1}/\Br K_1$ is contained in $\beta(\kk(B_{K_1})_\calE^{\times})$, meaning that $\calB - \beta(\tilde \ell) = \beta(\ell')$ for some $\ell' \in  \kk(B_{K_1})_\calE^{\times}$. Since both $\ell'$ and $\tilde\ell$ are in $ \kk(B_{K_1})_{\calE}^{\times}$, we then have $\calB =\beta(\tilde \ell \ell')=: \beta(\ell_{\calB})$ for some $\ell_{\calB} \in \kk(B_{K_1})_{\calE}$. Furthermore, since $\calB$ is $\Gal(K_1/K_0)$-invariant and is equal to $\beta(\tilde\ell)$ modulo $\Br_1 U_{K_1}$, Theorem~\ref{thm:Br1vsBr} implies that 
    \begin{enumerate}
        \item[(a)] the class of $\ell_{\calB}$ in $\kk(B_{K_1})/j(\Pic \Yabc_{, K_1})K_1^{\times}\kk(B_{K_1})^{\times2}$ must be $\Gal(K_1/K_0)$-invariant, and
        \item[(b)] $\ell_{\calB}\tilde\ell^{-1} \in j\left((\Pic \Yabcbar/(\pi^*\Pic S + 2\Pic \Yabcbar))^{G_{K_1}}\right)K_1^{\times}\kk(B_{K_1})^{\times2}$.
    \end{enumerate}
    
    An inspection of Table~\ref{table:Galois} reveals that $\Pic \Yabc_{, K_1} = \Pic S_{K_1}$, so $j(\Pic \Yabc_{, K_1}) \subset K_1^{\times}\kk(B_{K_1})^{\times2}$.  In addition, Lemma~\ref{lem:GaloisPicYbar} shows that every element of $j\left((\Pic \Yabcbar/(\pi^*\Pic S + 2\Pic \Yabcbar))^{G_{K_1}}\right)$ is $\Gal(K_0(\theta_0, \sqrt{ab})/K_0)$-invariant.  Thus, conditions $(a)$ and $(b)$ imply that 
    \[
    \tilde\ell\in \left(\frac{\kk(B_{K_1})^{\times}}
	{K_1^{\times}\kk(B_{K_1})^{\times2}}\right)^{
	\Gal(K_0(\theta_0, \sqrt{ab})/K_0)}.
    \]  Given our assumption on $\tilde\ell$, this results in a contradiction, as demonstrated by Table~\ref{table:Galois}.
    \qed